\documentclass[12pt]{article}

\usepackage{latexsym,amsmath,amscd,amssymb,graphics,float}
\usepackage{enumerate}

\usepackage{graphicx,lscape}

\usepackage[colorlinks]{hyperref}
\usepackage{url}

\begin{document}

\newenvironment{proof}[1][Proof]{\textbf{#1.} }{\ \rule{0.5em}{0.5em}}

\newtheorem{theorem}{Theorem}[section]
\newtheorem{definition}[theorem]{Definition}
\newtheorem{lemma}[theorem]{Lemma}
\newtheorem{remark}[theorem]{Remark}
\newtheorem{proposition}[theorem]{Proposition}
\newtheorem{corollary}[theorem]{Corollary}
\newtheorem{example}[theorem]{Example}

\numberwithin{equation}{section}
\newcommand{\ep}{\varepsilon}
\newcommand{\R}{{\mathbb  R}}
\newcommand\C{{\mathbb  C}}
\newcommand\Q{{\mathbb Q}}
\newcommand\Z{{\mathbb Z}}
\newcommand{\N}{{\mathbb N}}

\newcommand{\bfi}{\bfseries\itshape}

\newsavebox{\savepar}
\newenvironment{boxit}{\begin{lrbox}{\savepar}
\begin{minipage}[b]{15.5cm}}{\end{minipage}\end{lrbox}
\fbox{\usebox{\savepar}}}

\title{{\bf A stability criterion for non-degenerate equilibrium states of completely integrable systems}}
\author{R\u{a}zvan M. Tudoran}

\date{}
\maketitle \makeatother

\begin{abstract}
We provide a criterion in order to decide the stability of non-degenerate equilibrium states of completely integrable systems. More precisely, given a Hamilton-Poisson realization of a completely integrable system generated by a smooth $n-$ dimensional vector field, $X$, and a non-degenerate regular (in the Poisson sense) equilibrium state, $\overline{x}_e$, we define a scalar quantity, $\mathcal{I}_{X}(\overline{x}_e)$, whose sign determines the stability of the equilibrium. Moreover, if $\mathcal{I}_{X}(\overline{x}_e)>0$, then around $\overline{x}_e$, there exist one-parameter families of periodic orbits shrinking to $\{\overline{x}_e \}$, whose periods approach $2\pi/\sqrt{\mathcal{I}_{X}(\overline{x}_e)}$ as the parameter goes to zero. The theoretical results are illustrated in the case of the Rikitake dynamical system.
\end{abstract}

\medskip

\textbf{AMS 2010}: 37J25; 37J45; 37J05; 86A25.


\section{Introduction}
\label{section:one}
The aim of this article is to provide a criterion in order to decide the stability of non-degenerate equilibrium states of completely integrable systems. More precisely, given a Hamiltonian realization (of Poisson type) of a completely integrable system generated by a smooth $n-$dimensional vector field, $X$, and a non-degenerate regular (in the Poisson sense) equilibrium state, $\overline{x}_e$, we define a scalar quantity, $\mathcal{I}_{X}(\overline{x}_e)$, whose sign determines the stability of $\overline{x}_e$, i.e., \textit{if $\mathcal{I}_{X}(\overline{x}_e)>0$ then $\overline{x}_e$ is Lyapunov stable, whereas if $\mathcal{I}_{X}(\overline{x}_e)< 0$ then $\overline{x}_e$ is unstable}. Moreover, as the characteristic polynomial of the linearization of $X$ at $\overline{x}_e$, $\mathfrak{L}^{X}(\overline{x}_e)$, is given by $p_{\mathfrak{L}^{X}(\overline{x}_e)}(\mu)=(-\mu)^{n-2}\cdot\left( \mu^2 + \mathcal{I}_{X}(\overline{x}_e) \right)$, it follows that $\mathcal{I}_{X}(\overline{x}_e)$ depends only on $X$ and $\overline{x}_e$, and not on the Hamiltonian realization. Also, if we denote by $\Sigma_{\overline{x}_e}$, the symplectic leaf (passing through $\overline{x}_e$) of the Poisson configuration manifold of the Hamiltonian realization, then the sign of $\mathcal{I}_{X}(\overline{x}_e)$ determines again the stability of $\overline{x}_e$, this time regarded as an equilibrium state of the restricted vector field $X|_{\Sigma_{\overline{x}_e}}$. Moreover, if $\mathcal{I}_{X}(\overline{x}_e)>0$, then there exists $\varepsilon_{0}>0$ and a one-parameter family of periodic orbits of $X|_{\Sigma_{\overline{x}_e}}$ (and hence of $X$ too), $\left\{\gamma_{\varepsilon}\right\}_{0<\varepsilon\leq\varepsilon_0}\subset \Sigma_{\overline{x}_e}$, that shrink to $\{\overline{x}_e \}$ as $\varepsilon\rightarrow 0$, with periods $T_{\varepsilon}\rightarrow{\frac{2\pi}{\sqrt{\mathcal{I}_{X}(\overline{x}_e)}}}$ as $\varepsilon\rightarrow 0$. Also, the set $\{\overline{x}_e\}\cup\bigcup_{0<\varepsilon < \varepsilon_0} \gamma_{\varepsilon}$ represents the connected component of  $\Sigma_{\overline{x}_e}\setminus \gamma_{\varepsilon_{0}}$, which contains the equilibrium point $\overline{x}_e$. Note that by choosing a different Hamiltonian realization of the completely integrable system, for which $\overline{x}_e$ is also a non-degenerate regular equilibrium point, we obtain the existence of a different family of periodic orbits with the same properties, this time the orbits being located on the regular symplectic leaf (passing through $\overline{x}_e$) corresponding to the Poisson configuration manifold associated to this specific Hamiltonian realization. On the applicative level, all theoretical results are illustrated in the case of the Rikitake dynamical system.

More precisely, the structure of the article is the following: the second section contains a a short introduction to the geometry associated to a general completely integrable system. More precisely, using the property that any completely integrable system admits Hamiltonian realizations of Poisson type, we briefly present the associated Poisson geometry, and its relations with the dynamics generated by the system. The aim of the third section is to characterize the set of equilibrium states of a general completely integrable system, and also to analyze the geometric and analytic properties of certain subsets of equilibria, naturally associated with the Poisson geometry of the Hamiltonian realizations of the system. In fourth section of the article we define the scalar quantity $\mathcal{I}_{X}(\overline{x}_e)$, and analyze its main geometric and analytic properties. The fifth section is the main part of this article and contains the main result, which provides a criterion to test the stability of non-degenerate regular equilibrium states of Hamiltonian realizations of completely integrable systems. The aim of the sixth section is to give a criterion to decide leafwise stability of non-degenerate regular equilibria of Hamiltonian realizations of completely integrable systems, and also to study the local existence of periodic orbits. In the last section, we illustrate the main theoretical results in the case of a concrete example coming from geophysics, namely, the so called Rikitake two-disc dynamo system.

\section{A geometric formulation of completely integrable systems}

The aim of this section is to give a short introduction to the geometry associated to a general completely integrable system. More precisely, using the property that any completely integrable system admits Hamiltonian realizations of Poisson type (see e.g., \cite{tudoran}), we present the associated Poisson geometry, and its relations with the dynamics generated by the system.

In order to do that, let us start by recalling from \cite{tudoran} the Hamiltonian realization procedure of a completely integrable system. For similar Hamilton-Poisson and respectively Nambu-Poisson formulations of completely integrable systems, see e.g., \cite{abrahammarsden}, \cite{marsdenratiu}, \cite{nambu1}, \cite{nambu2}, \cite{ratiurazvan}.

Recall that a \textit{completely integrable system} is a $\mathcal{C}^{\infty}$ differential system defined on an open subset $\Omega\subseteq\R^n$,
\begin{equation}\label{sys}
\left\{\begin{array}{l}
\dot x_{1}=X_1(x_1,\dots,x_n)\\
\dot x_{2}=X_2(x_1,\dots,x_n)\\
\cdots\\
\dot x_{n}=X_n(x_1,\dots,x_n),\\
\end{array}\right.
\end{equation}
(where $X_1,X_2,\dots,X_n\in\mathcal{C}^{\infty}(\Omega,\R)$ are smooth functions), which admits a set of smooth first integrals, $C_1,\dots,C_{n-2},C_{n-1}:\Omega\rightarrow \R$, functionally independent almost everywhere with respect to the $n-$dimensional Lebesgue measure.

Since the smooth functions $C_1,\dots,C_{n-2},C_{n-1}:\Omega\rightarrow \R$ are constants of motion of the vector field $X=X_1\dfrac{\partial}{\partial{x_1}}+\dots+X_n \dfrac{\partial}{\partial{x_n}}\in\mathfrak{X}(\Omega)$, it follows that for each $i\in\{1,\dots,n-1\}$,
$$
\langle\nabla C_i(\overline{x}),X(\overline{x})\rangle =0,
$$
for every $\overline{x}=(x_1,\dots,x_n)\in\Omega$ (where $\langle \cdot,\cdot\rangle$ is the canonical inner product on $\R^n$, and $\nabla$ stands for the gradient with respect to $\langle \cdot,\cdot\rangle$).

Hence, as shown in \cite{tudoran}, the vector field $X$ is proportional with the vector field $\star (\nabla C_1\wedge\dots\wedge \nabla C_{n-1})$, where $\star$ stands for the Hodge star operator for multi-vector fields. It may happen that the domain of definition of the proportionality rescaling function, is a proper subset of $\Omega$. In order to simplify the notations, we shall work in the sequel on this subset, which will be also denoted by $\Omega$.  

Consequently, the vector field $X$ admits the local expression
\begin{equation}\label{sywedge}
X = (-\nu) \star(\nabla C_1 \wedge \dots \wedge\nabla C_{n-1}),
\end{equation}
where $\nu\in\mathcal{C}^{\infty}(\Omega,\R)$ stands for the rescaling function. Note that each permutation of the first integrals $C_1,\dots, C_{n-1}$ within the wedge product $\nabla C_1 \wedge \dots \wedge\nabla C_{n-1}$, together with a possible change of sign of the rescaling function, give rise to another realization of the vector field $X$ of type \eqref{sywedge}.

Let us fix from now on the realization \eqref{sywedge} of the vector field $X$. Following \cite{tudoran}, we shall express the vector field $X$ as a Hamilton-Poisson vector field, $X_H \in\mathfrak{X}(\Omega)$, with respect to the Hamiltonian function $H:=C_{n-1}$, and the Poisson bracket given by
$$
\{f,g\}_{\nu;C_1,\dots,C_{n-2}}\cdot \mathrm{d}x_1\wedge\dots\wedge \mathrm{d}x_n:=\nu \cdot \mathrm{d}C_1\wedge\dots \wedge\mathrm{d}C_{n-2}\wedge \mathrm{d}f\wedge \mathrm{d}g,
$$
for every $f,g\in\mathcal{C}^{\infty}(\Omega,\mathbb{R})$. Recall that a general \textit{Poisson bracket} on $\Omega$ is a bilinear map $\{\cdot,\cdot\}:\mathcal{C}^{\infty}(\Omega,\mathbb{R})\times \mathcal{C}^{\infty}(\Omega,\mathbb{R})\rightarrow \mathcal{C}^{\infty}(\Omega,\mathbb{R})$ that defines a Lie algebra structure on $\mathcal{C}^{\infty}(\Omega,\mathbb{R})$ and moreover is a derivation in each entry. A pair $(\Omega,\{\cdot,\cdot\})$, where $\{\cdot,\cdot\}$ is a Poisson bracket on $\Omega$, is called a \textit{Poisson manifold}. 

In order to have a self-contained presentation, let us briefly recall the main ingredients that come along with the Hamiltonian realization procedure. More precisely, recall first that the derivation property of the Poisson bracket implies that for any two functions $f,g\in\mathcal{C}^{\infty}(\Omega,\mathbb{R})$, the bracket $\{f,g\}_{\nu;C_1,\dots,C_{n-2}}(\overline{x})$ evaluated at an arbitrary point $\overline{x}\in\Omega$, depends on $f$ only through $\mathrm{d}f(\overline{x})$. This property allows us to define a contravariant antisymmetric $2-$tensor,  $\Pi_{\nu;C_1,\dots,C_{n-2}}$, given by 
$$
\Pi_{\nu;C_1,\dots,C_{n-2}}(\overline{x})(\alpha_{\overline{x}},\beta_{\overline{x}})=\{f,g\}_{\nu;C_1,\dots,C_{n-2}}(\overline{x}),
$$
where $\mathrm{d}f(\overline{x})=\alpha_{\overline{x}}\in T^{*}_{\overline{x}}\Omega\cong \mathbb{R}^n$ and $\mathrm{d}g(\overline{x})=\beta_{\overline{x}}\in T^{*}_{\overline{x}}\Omega \cong \mathbb{R}^n$. This tensor is called the \textit{Poisson tensor} or the \textit{Poisson structure} generated by the Poisson bracket $\{\cdot,\cdot\}_{\nu;C_1,\dots,C_{n-2}}$.

The vector bundle map $\Pi_{\nu;C_1,\dots,C_{n-2}}^{\sharp}:T^{*}\Omega\to T\Omega$, naturally associated to $\Pi_{\nu;C_1,\dots,C_{n-2}}$, is given for each $\overline{x}\in \Omega$ by the linear map $\left(\Pi_{\nu;C_1,\dots,C_{n-2}}^{\sharp}\right)_{\overline{x}}:T_{\overline{x}}^{*}\Omega\to T_{\overline{x}}\Omega$,
defined by the equality
\begin{equation*}
\Pi_{\nu;C_1,\dots,C_{n-2}}(\overline{x})(\alpha_{\overline{x}},\beta_{\overline{x}})=\left<
\alpha_{\overline{x}},\left(\Pi_{\nu;C_1,\dots,C_{n-2}}^{\sharp}\right)_{\overline{x}}(\beta_{\overline{x}})\right>, ~ (\forall) \alpha_{\overline{x}},\beta_{\overline{x}} \in T_{\overline{x}}^{*}\Omega.
\end{equation*}

The above defined bundle map induces for each $H\in\mathcal{C}^{\infty}(\Omega,\mathbb{R})$, a smooth vector field, 
$$X_H:=\Pi_{\nu;C_1,\dots,C_{n-2}}^{\sharp}(\mathrm{d}H),$$
called the \textit{Hamiltonian vector field} associated to $H$. 

As a differential operator, the Hamiltonian vector field $X_H$ acts on an arbitrary smooth function $f\in \mathcal{C}^{\infty}(\Omega,\R)$ as follows 
$$X_{H}(f)=\{f,H\}_{\nu;C_1,\dots,C_{n-2}}\in \mathcal{C}^{\infty}(\Omega,\R).$$

Consequently, a smooth function $f\in\mathcal{C}^{\infty}(\Omega,\mathbb{R})$ is a first integral of $X_H$ if and only if $\{f,H\}_{\nu;C_1,\dots,C_{n-2}}=0$.

Summarizing, we obtained that the completely integrable system \eqref{sys} admits the (local) Hamiltonian realization $\left(\Omega,\{\cdot,\cdot\}_{\nu;C_1,\dots,C_{n-2}}, H:=C_{n-1} \right)$, defined on the Poisson manifold $\left(\Omega,\{\cdot,\cdot\}_{\nu;C_1,\dots,C_{n-2}} \right)$. More precisely, the dynamical system \eqref{sys} might be equivalently written as
\begin{equation}\label{systy}
\left\{\begin{array}{l}
\dot x_{1}=\{x_1,H\}_{\nu;C_1,\dots,C_{n-2}}\\
\dot x_{2}=\{x_2,H\}_{\nu;C_1,\dots,C_{n-2}}\\
\cdots \\
\dot x_{n}=\{x_n,H\}_{\nu;C_1,\dots,C_{n-2}}.\\  
\end{array}\right.
\end{equation}

Otherwise stated, using the definition of the Poisson bracket, the restriction to $\Omega$ of the components $X_i$ of the vector field $X$ (which generates the system \eqref{sys}), admit the formulation
$$X_i=\nu \cdot \dfrac{\partial(C_1,\dots,C_{n-2},x_i,H)}{\partial(x_1,\dots,x_n)},$$
for every $i\in\{1,\dots,n\}$.

So far we have seen some dynamical implications induced by the Poisson structure $\Pi_{\nu;C_1,\dots,C_{n-2}}$. Next, we shall analyze some of the main geometrical features of the ambient space $\Omega$, induced by the existence of the Poisson structure $\Pi_{\nu;C_1,\dots,C_{n-2}}$. First of all, recall that the expression of the Poisson tensor relative to a local coordinates system, $(x_1,...,x_n)$, is given by the bi-vector field
$$
\Pi_{\nu;C_1,\dots,C_{n-2}} = \sum_{1\leq k < \ell \leq n}\Pi_{\nu;C_1,\dots,C_{n-2}}^{k \ell}(x_1, \ldots, x_n)
\frac{\partial}{\partial x_k} \wedge \frac{\partial}{\partial
x_\ell}\,,
$$
where $\Pi_{\nu;C_1,\dots,C_{n-2}}^{k \ell}(x_1, \ldots, x_n):=\{x_k,x_\ell\}_{\nu;C_1,\dots,C_{n-2}}$.

If there is no danger of confusion, the skew-symmetric matrix $$\Pi_{\nu;C_1,\dots,C_{n-2}}:=[ \Pi_{\nu;C_1,\dots,C_{n-2}}^{k \ell}(x_1, \ldots, x_n)]_{1\leq k, \ell \leq n}$$ will also be called Poisson structure.

Consequently, using the local matrix expression of the Poisson tensor, the local expression of the Poisson bracket $\{f,g\}_{\nu;C_1,\dots,C_{n-2}}$ of two arbitrary smooth functions $f,g\in \mathcal{C}^{\infty}(\Omega,\mathbb{R})$, becomes,
\begin{equation}\label{pbloc}
\{f,g\}_{\nu;C_1,\dots,C_{n-2}} =(\nabla f)^{\top}\Pi_{\nu;C_1,\dots,C_{n-2}}\nabla g.
\end{equation}

The existence of the Poisson structure $\Pi_{\nu;C_1,\dots,C_{n-2}}$ induces a partition of $\Omega$ in two subsets, i.e., the set of regular points and its complement, the set of singular points. More precisely, a point $\overline{x}_0 \in\Omega$ is called \textit{regular point} of the Poisson structure $\Pi_{\nu;C_1,\dots,C_{n-2}}$ if there exists $U\subseteq\Omega$, an open neighborhood of $\overline{x}_0$, such that the rank of $\Pi_{\nu;C_1,\dots,C_{n-2}}$ is constant for every $\overline{x}\in U$, i.e., the rank of the the linear map $$\left(\Pi_{\nu;C_1,\dots,C_{n-2}}^{\sharp}\right)_{\overline{x}}:T_{\overline{x}}^{*}\Omega\to T_{\overline{x}}\Omega$$ is constant for every $\overline{x}\in U$, or equivalently, $$\operatorname{rank} [\Pi_{\nu;C_1,\dots,C_{n-2}}^{k \ell}(\overline{x})]_{1\leq k, \ell \leq n}=\operatorname{rank} [\Pi_{\nu;C_1,\dots,C_{n-2}}^{k \ell}(\overline{x}_0)]_{1\leq k, \ell \leq n},$$ for every $\overline{x}\in U$. The set of regular points of $\Omega$ is denoted by $\Omega_{reg}$. Due to skew-symmetry of the Poisson structure, the rank of every point is an even number. Moreover, the lower semi-continuity of the rank function, (i.e., for each point $\overline{x}_0 \in\Omega$ there exists an open neighborhood $U\subseteq \Omega$ of $\overline{x}_0$ such that $\operatorname{rank}\Pi_{\nu;C_1,\dots,C_{n-2}}(\overline{x})\geq \operatorname{rank}\Pi_{\nu;C_1,\dots,C_{n-2}}(\overline{x}_0)$ for all $\overline{x}\in U$) implies that the set of regular points, $\Omega_{reg}$, is an open dense subset of $\Omega$, and $\Omega_{sing}:=\Omega\setminus\Omega_{reg}$, the set of \textit{singular} points, is a closed nowhere dense subset of $\Omega$.   
\begin{remark}\label{regpts}
From the definition of the Poisson bracket it follows that for each $\overline{x}\in\Omega$, $\operatorname{rank}\Pi_{\nu;C_1,\dots,C_{n-2}}(\overline{x})\in\{0,2\}$. More precisely, $\operatorname{rank}\Pi_{\nu;C_1,\dots,C_{n-2}}(\overline{x})=0$ either if $\overline{x}\in\Omega$ such that $\nu(\overline{x})=0$, or if the vectors $\nabla C_1 (\overline{x}),\dots,\nabla C_{n-2}(\overline{x})$ are linearly dependent (i.e., $\nabla C_1 (\overline{x})\wedge \dots \wedge \nabla C_{n-2}(\overline{x})=0$), and $\operatorname{rank}\Pi_{\nu;C_1,\dots,C_{n-2}}(\overline{x})=2$, otherwise.   
\end{remark}

Let us provide now a characterization of the regular points of $\Omega$ with respect to the Poisson structure $\Pi_{\nu;C_1,\dots,C_{n-2}}$. In order to do that, recall from the definition of the completely integrable system \eqref{sys} that the gradient vector fields, $\nabla C_1, \dots, \nabla C_{n-2}$, are pointwise linearly independent almost everywhere (with respect to the $n-$dimensional Lebesgue measure). Since the rescaling function $\nu$ is supposed to be a generic smooth function, it will be nonzero almost everywhere. Consequently, from Remark \eqref{regpts} we get that $\overline{x}_0 \in\Omega_{reg}$ if and only if $\operatorname{rank}\Pi_{\nu;C_1,\dots,C_{n-2}}(\overline{x}_0)=\max_{\{\overline{x}\in\Omega\}}\operatorname{rank}\Pi_{\nu;C_1,\dots,C_{n-2}}(\overline{x})$. Recall that the set of maximal rank points is an open set which generally need not be dense (e.g., if $\nu$ is a smooth function such that $\nu^{-1}(\{0\})$ contains a proper open subset of $\Omega$). In general, the set of maximal rank points is only included in the set of regular points. 

Hence, we obtain the following characterization of a regular point.
\begin{remark}\label{regulpts}
A point $\overline{x}_0 \in\Omega$ is regular if and only if $\operatorname{rank}\Pi_{\nu;C_1,\dots,C_{n-2}}(\overline{x}_0)=2$, or equivalently, $\overline{x}_0 \in\Omega_{reg}$ if and only if $\nu(\overline{x}_0)\neq 0$ and $\nabla C_1 (\overline{x}_0)\wedge \dots \wedge \nabla C_{n-2}(\overline{x}_0)\neq 0$.
\end{remark}

In the following we briefly present some of the main geometrical properties of the Poisson manifold $(\Omega,\{\cdot,\cdot\}_{\nu;C_1,\dots,C_{n-2}})$. In order to do that, let us recall first the definition of a Casimir function. More exactly, a smooth function $C\in\mathcal{C}^{\infty}(\Omega,\mathbb{R})$ which verifies that $\{f,C\}_{\nu;C_1,\dots,C_{n-2}}=0$, for every $f\in\mathcal{C}^{\infty}(\Omega,\mathbb{R})$, is called a \textit{Casimir function} of the Poisson structure $\Pi_{\nu;C_1,\dots,C_{n-2}}$. Note that $C_1,\dots,C_{n-2}$, form a \textit{complete set} of Casimir functions of the Poisson bracket $\{\cdot,\cdot\}_{\nu;C_1,\dots,C_{n-2}}$, i.e., each Casimir function of the Poisson structure $\Pi_{\nu;C_1,\dots,C_{n-2}}$, might be written as a functional combination of $C_1,\dots,C_{n-2}$.

The relation \eqref{pbloc} implies that a smooth function $C\in\mathcal{C}^{\infty}(\Omega,\mathbb{R})$ is a Casimir function if and only if 
\begin{equation}\label{csmw}
\nabla C(\overline{x})\in\ker \Pi_{\nu;C_1,\dots,C_{n-2}}(\overline{x}), ~(\forall)\overline{x}\in\Omega.
\end{equation}

Let us consider $\overline{x}_0 \in \Omega_{reg}$. Hence, the vectors $\nabla C_1 (\overline{x}_0),\dots,\nabla C_{n-2}(\overline{x}_0)$ are linearly independent, and consequently  $$\dim_{\mathbb{R}}(\operatorname{span}_{\mathbb{R}}\{\nabla C_{1}(\overline{x}_0),\dots, \nabla C_{n-2}(\overline{x}_0)\})=n-2.$$ 

Since $\overline{x}_0$ is a regular point, it follows that $\operatorname{rank}\Pi_{\nu;C_1,\dots,C_{n-2}}(\overline{x}_0)=2$, and hence one obtains that
\begin{align*}
\begin{split}
\dim_{\mathbb{R}}(\ker\Pi_{\nu;C_1,\dots,C_{n-2}}(\overline{x}_0))&=n-\dim_{\mathbb{R}}(\operatorname{Im}\Pi_{\nu;C_1,\dots,C_{n-2}}(\overline{x}_0))=n-\operatorname{rank}\Pi_{\nu;C_1,\dots,C_{n-2}}(\overline{x}_0)\\
&=n-2.
\end{split}
\end{align*}

Using the relation \eqref{csmw} associated to $C_1,\dots,C_{n-2}$, the following inclusion of vector subspaces holds true for every $\overline{x}\in\Omega$:
\begin{equation}\label{inclus}
\operatorname{span}_{\mathbb{R}}\{\nabla C_{1}(\overline{x}),\dots, \nabla C_{n-2}(\overline{x})\}\subseteq \ker\Pi_{\nu;C_1,\dots,C_{n-2}}(\overline{x}).
\end{equation}

Consequently, due to the equality of dimensions, the inclusion \eqref{inclus} becomes an equality for $\overline{x}=\overline{x}_0$. 

Moreover, we obtained another characterization of regular points. More precisely, a point $\overline{x}\in\Omega$ is a regular point if and only if
\begin{equation}\label{inclus1}
\operatorname{span}_{\mathbb{R}}\{\nabla C_{1}(\overline{x}),\dots, \nabla C_{n-2}(\overline{x})\}= \ker\Pi_{\nu;C_1,\dots,C_{n-2}}(\overline{x}).
\end{equation}

So far we pointed out some geometrical properties of the kernel of the linear maps $\left(\Pi_{\nu;C_1,\dots,C_{n-2}}^{\sharp}\right)_{\overline{x}}:T_{\overline{x}}^{*}\Omega\to T_{\overline{x}}\Omega$, for $\overline{x}\in\Omega$. In the following we shall analyze the geometrical properties of the image of these maps. In order to do that, note that for each $\overline{x}\in\Omega$, the image of the linear map $\left(\Pi_{\nu;C_1,\dots,C_{n-2}}^{\sharp}\right)_{\overline{x}}:T_{\overline{x}}^{*}\Omega\to T_{\overline{x}}\Omega$ is a vector subspace $S_{\overline{x}}\subseteq T_{\overline{x}}\Omega \cong \mathbb{R}^n$ of dimension equal to $\operatorname{rank}\Pi_{\nu;C_1,\dots,C_{n-2}}(\overline{x})$. The collection of these vector subspaces, for $\overline{x}\in\Omega$, forms a smooth generalized distribution, called the \textit{characteristic distribution} of the Poisson structure $\Pi_{\nu;C_1,\dots,C_{n-2}}$. Since the rank of the characteristic distribution at $\overline{x}\in\Omega$, coincides with $\operatorname{rank}\Pi_{\nu;C_1,\dots,C_{n-2}}(\overline{x})$, from the Remark \eqref{regpts} one obtains that the rank of the distribution at regular points is two, while the rank at the singular points is zero. The \textit{symplectic stratification theorem} states that the characteristic distribution is integrable, and the leaves of the induced foliation are symplectic manifolds (i.e., manifolds which admit some smooth, nondegenerate, closed $2-$form). If one denotes by $\Sigma_{\overline{x}}\subset \Omega$ the leaf through the point $\overline{x}\in\Omega$, then the restriction of the Poisson bracket $\{\cdot,\cdot\}_{\nu;C_1,\dots,C_{n-2}}$ to $\Sigma_{\overline{x}}$, induces on $\Sigma_{\overline{x}}$ a symplectic form, $\omega_{\Sigma_{\overline{x}}}$, defined for each pair of vectors $v_{\overline{x}},w_{\overline{x}}\in T_{\overline{x}}\Sigma_{\overline{x}}=S_{\overline{x}}$, by the formula
$$
\omega_{\Sigma_{\overline{x}}}(\overline{x})(v_{\overline{x}},w_{\overline{x}}):=\Pi_{\nu;C_1,\dots,C_{n-2}}(\overline{x})(\alpha_{\overline{x}},\beta_{\overline{x}}),
$$
where $\alpha_{\overline{x}},\beta_{\overline{x}}\in T_{\overline{x}}^{*}\Omega$ are covectors corresponding to $v_{\overline{x}},w_{\overline{x}}\in T_{\overline{x}}\Sigma_{\overline{x}}=S_{\overline{x}}$, through the linear map $\left(\Pi_{\nu;C_1,\dots,C_{n-2}}^{\sharp}\right)_{\overline{x}}$. Note that the $2-$form $\omega_{\Sigma_{\overline{x}}}$ is closed, since the Poisson bracket $\{\cdot,\cdot\}_{\nu;C_1,\dots,C_{n-2}}$ verifies the Jacobi identity. As a set, the symplectic leaf $\Sigma_{\overline{x}}$ is given by those points of $\Omega$ which can be joined with $\overline{x}$ by a piecewise smooth path, consisting of smooth pieces of integral curves of Hamiltonian vector fields.

Concerning the partition of $\Omega$ induced by the Poisson structure $\Pi_{\nu;C_1,\dots,C_{n-2}}$, recall that both the set of regular points, as well as the set of singular points of the Poisson manifold $(\Omega,\{\cdot,\cdot\}_{\nu;C_1,\dots,C_{n-2}})$, are \textit{saturated} subsets of $\Omega$, i.e., they are both some union of symplectic leaves. Since each symplectic leaf is a connected set, it follows that if $\overline{x}$ is a regular point, then so are all points in the corresponding sympletic leaf, $\Sigma_{\overline{x}}$, i.e., if $\overline{x}\in\Omega_{reg}$, then $\Sigma_{\overline{x}}\subseteq\Omega_{reg}$. The leaves through regular points are called \textit{regular leaves}, while the rest of them are called \textit{singular leaves}. Note that each regular leaf of the Poisson manifold $(\Omega,\{\cdot,\cdot\}_{\nu;C_1,\dots,C_{n-2}})$ is two-dimensional, while each singular leaf of $(\Omega,\{\cdot,\cdot\}_{\nu;C_1,\dots,C_{n-2}})$ is zero-dimensional. 

More precisely, if one denotes by $C:=(C_1,\dots,C_{n-2}): \Omega \longrightarrow \mathbb{R}^{n-2}$ the map generated by the Casimir functions $C_1, \dots, C_{n-2}$, then the regular leaves of the symplectic foliation of the Poisson manifold  $(\Omega,\{\cdot,\cdot\}_{\nu;C_1,\dots,C_{n-2}})$ are the connected components of the two-dimensional manifolds given by $C^{-1}(\{(c_1,\dots,c_{n-2})\})\setminus Z(\nu)$, if $(c_1,\dots,c_{n-2})$ is a regular value of $C$, or given by $C^{-1}(\{(c_1,\dots,c_{n-2})\})\setminus\{Z(\nu)\cup \operatorname{Crit}(C)\}$, if $(c_1,\dots,c_{n-2})$ is a critical value of $C$, where $\operatorname{Crit}(C)\subset \Omega$ stands for the set of critical points of $C$, and $Z(\nu)$ is the set of zeros of the generic smooth function $\nu$. Moreover, the singular leaves are the zero-dimensional manifolds consisting each of some single point of the set $Z(\nu)\cup \operatorname{Crit}(C)$.

\section{The set of equilibrium states of a completely integrable system}

The aim of this section is to characterize the set of equilibrium states of a general completely integrable system, and also to analyze the geometric and analytic properties of certain subsets of equilibria, naturally associated with the Poisson geometry of the Hamiltonian realizations of the system.

In order to do that, let us recall first the relation \eqref{sywedge}, which provides a local expression of the vector field $X\in\mathfrak{X}(U)$ associated to the completely integrable system \eqref{sys}, i.e.,
\begin{equation}\label{frmty}
X = (-\nu) \star(\nabla C_1 \wedge \dots \wedge\nabla C_{n-1}),
\end{equation}
where $\nu\in\mathcal{C}^{\infty}(\Omega,\R)$ is a rescaling function. Next result gives a characterization of the \textit{equilibrium states} of the vector field $X$, i.e., the solutions of the equation $X(\overline{x})=0$, $\overline{x}\in\Omega$.

\begin{proposition}\label{equilibria} 
The equilibrium states of the integrable system \eqref{sys} are the elements of the set 
$$
\mathcal{E}^{X}:=\{\overline{x} \in \Omega \ | \ \nu (\overline{x})=0\}\cup\{\overline{x} \in \Omega \ | \ \nabla C_1 (\overline{x})\wedge \dots \wedge \nabla C_{n-1} (\overline{x}) =0 \}.
$$
\end{proposition}
\begin{proof}
Using the expression \eqref{frmty}, we obtain that $\overline{x}\in\Omega$ is an equilibrium state of the completely integrable system \eqref{sys} if and only if $\|X(\overline{x})\|=0$. The conclusion follows taking into account that $\|X(\overline{x})\|=|\nu (\overline{x})|\cdot\|\star(\nabla C_1 (\overline{x})\wedge \dots \wedge \nabla C_{n-1} (\overline{x}))\|$, and
$$
\|\star(\nabla C_1 (\overline{x})\wedge \dots \wedge \nabla C_{n-1} (\overline{x}))\|=\|\nabla C_1 (\overline{x})\wedge \dots \wedge \nabla C_{n-1} (\overline{x})\|_{n-1},
$$ 
where $\|\cdot\|_{n-1}:=\sqrt{\langle\cdot,\cdot\rangle_{n-1}}$ stands for the $(n-1)-$volume of decomposable $(n-1)-$vectors, and $\langle\cdot,\cdot\rangle_{n-1}$ denotes the inner product defined on arbitrary pairs of decomposable $(n-1)-$vectors, $u_1\wedge\dots\wedge u_{n-1}, v_1\wedge\dots\wedge v_{n-1}\in\Lambda^{n-1}\mathbb{R}^{n}$, by 
$$
\langle u_1\wedge\dots\wedge u_{n-1}, v_1\wedge\dots\wedge v_{n-1}\rangle _{n-1}:=\det([\langle u_i,v_j\rangle]_{1\leq i,j \leq {n-1}}).
$$
\end{proof}

In the sequel we shall analyze the local dynamical behavior of the completely integrable system \eqref{systy} around equilibrium states for which $\nu (\overline{x})\neq 0$, and moreover, there exists an $i\in\{1,\dots, n-1 \}$ such that 
$$
\nabla C_1 (\overline{x})\wedge \dots \wedge \nabla C_{i-1} (\overline{x})\wedge \widehat{\nabla C_{i} (\overline{x})}\wedge \nabla C_{i+1} (\overline{x})\wedge\dots\wedge\nabla C_{n-1} (\overline{x}) \neq 0,
$$
where " $\widehat{\cdot}$ " means that the indicated element is omitted. 

By eventually relabeling the first integrals $C_1,\dots, C_{n-1}$, we suppose that the above equilibrium states are elements of the set 
\begin{equation}\label{ndgeq}
\mathcal{E}^{C_{n-1}}_{C_1 ,\dots, C_{n-2}}:=\{\overline{x} \in \Omega \ | \ \nu (\overline{x})\neq 0, \ \nabla C_1 (\overline{x})\wedge \dots \wedge \nabla C_{n-2} (\overline{x}) \neq 0, \ \nabla C_1 (\overline{x})\wedge \dots \wedge \nabla C_{n-1} (\overline{x}) =0\}.
\end{equation} 

In order to give a geometric description of the set $\mathcal{E}^{C_{n-1}}_{C_1 ,\dots, C_{n-2}}$, let us recall from \eqref{systy} that the completely integrable system \eqref{sys} admits the Hamiltonian realization $$(\Omega,\{\cdot,\cdot\}_{\nu;C_1,\dots,C_{n-2}},H=C_{n-1})$$ modeled on the Poisson manifold $(\Omega,\{\cdot,\cdot\}_{\nu;C_1,\dots,C_{n-2}})$. Consequently, using the characterization of regular points of $(\Omega,\{\cdot,\cdot\}_{\nu;C_1,\dots,C_{n-2}})$ given in Remark \eqref{regulpts}, we obtain the following description of the regular equilibrium states of the integrable system \eqref{sys}, viewed as the Hamiltonian system \eqref{systy}.
\begin{proposition}\label{regeq} 
The set of regular equilibrium states of the completely integrable system \eqref{sys}, realized as the Hamiltonian system $(\Omega,\{\cdot,\cdot\}_{\nu;C_1,\dots,C_{n-2}},H=C_{n-1})$ modeled on the Poisson manifold $(\Omega,\{\cdot,\cdot\}_{\nu;C_1,\dots,C_{n-2}})$, is given by
\begin{equation*}
\mathcal{E}^{C_{n-1}}_{C_1 ,\dots, C_{n-2}}=\{\overline{x} \in \Omega \ | \ \nu (\overline{x})\neq 0, \ \nabla C_1 (\overline{x})\wedge \dots \wedge \nabla C_{n-2} (\overline{x}) \neq 0, \ \nabla C_1 (\overline{x})\wedge \dots \wedge \nabla C_{n-1} (\overline{x}) =0\}.
\end{equation*} 
\end{proposition}
\begin{proof}
Recall from Proposition \eqref{equilibria} that the set of equilibrium states of the integrable system \eqref{sys} is given by
$$
\mathcal{E}^{X}=\{\overline{x} \in \Omega \ | \ \nu (\overline{x})=0\}\cup\{\overline{x} \in \Omega \ | \ \nabla C_1 (\overline{x})\wedge \dots \wedge \nabla C_{n-1} (\overline{x}) =0 \}.
$$
Recall also from the Remark \eqref{regulpts} that a point $\overline{x} \in\Omega$ is regular with respect to the Poisson structure $\Pi_{\nu;C_1,\dots,C_{n-2}}$ if and only if $\nu(\overline{x})\neq 0$ and $\nabla C_1 (\overline{x})\wedge \dots \wedge \nabla C_{n-2}(\overline{x})\neq 0$. Consequently, we obtain that
$$
\mathcal{E}^{X} \cap \Omega_{reg}=\mathcal{E}^{C_{n-1}}_{C_1 ,\dots, C_{n-2}},
$$
and hence we get the conclusion.
\end{proof}

Let us fix now $\overline{x}_e \in \mathcal{E}^{C_{n-1}}_{C_1 ,\dots, C_{n-2}}$, a regular equilibrium state of the dynamical system \eqref{sys}, realized as the Hamiltonian system $(\Omega,\{\cdot,\cdot\}_{\nu;C_1,\dots,C_{n-2}},H=C_{n-1})$ . Then, using the characterization of regular equilibrium states, given in Proposition \eqref{regeq}, there exists $\overrightarrow{\lambda_{e}}:=(\lambda^{e}_1,\dots,\lambda^{e}_{n-2})\in\mathbb{R}^{n-2}$ such that $$\nabla C_{n-1}(\overline{x}_e)+\lambda^{e}_1 \nabla C_1 (\overline{x}_e)+\dots +\lambda^{e}_{n-2} \nabla C_{n-2} (\overline{x}_e)=0,$$ or equivalently, $\nabla F_{\overrightarrow{\lambda_{e}}}(\overline{x}_e)=0$, where $F_{\overrightarrow{\lambda_{e}}}:\Omega\rightarrow \mathbb{R}$ is the smooth function given by $F_{\overrightarrow{\lambda_{e}}}:=C_{n-1}+\lambda^{e}_1 C_{1} +\dots +\lambda^{e}_{n-2} C_{n-2}$.

\begin{definition}\label{ndgequil}
A regular equilibrium state $\overline{x}_e \in \mathcal{E}^{C_{n-1}}_{C_1 ,\dots, C_{n-2}}$ will be called \textbf{non-degenerate} if it is a non-degenerate critical point of $F_{\overrightarrow{\lambda_{e}}}$, i.e., $\nabla F_{\overrightarrow{\lambda_{e}}}(\overline{x}_e)=0$, and   $\det{(\operatorname{Hess}F_{\overrightarrow{\lambda_{e}}}(\overline{x}_e))}\neq 0$, where the linear map $\operatorname{Hess}F_{\overrightarrow{\lambda_{e}}}(\overline{x}_e) : \mathbb{R}^{n} \longrightarrow \mathbb{R}^{n}$, given by  $$\operatorname{Hess}F_{\overrightarrow{\lambda_{e}}}(\overline{x}_e)\cdot v := \mathrm{D}(\nabla F_{\overrightarrow{\lambda_{e}}})(\overline{x}_e)\cdot v, ~ (\forall)v\in \mathbb{R}^{n},$$ stands for the canonical Hessian operator on $(\mathbb{R}^{n},<\cdot,\cdot>)$.
\end{definition}

Next result provides a local property of the set of non-degenerate regular equilibrium points of the completely integrable system \eqref{sys}, viewed as the Hamiltonian dynamical system \eqref{systy}.
\begin{theorem}\label{ift}
Let $\overline{x}_e \in \mathcal{E}^{C_{n-1}}_{C_1 ,\dots, C_{n-2}}$ be a non-degenerate regular equilibrium state of the completely integrable system \eqref{systy}. Then there exist $V\subseteq \mathbb{R}^{n-2}$, an open neighborhood of $\overrightarrow{\lambda_{e}}$, $U\subseteq \Omega$, an open neighborhood of $\overline{x}_e$, and a smooth function $\overline{x}: V\rightarrow U$ such that $\overline{x}(\overrightarrow{\lambda_{e}})=\overline{x}_e$, and moreover, for each $\overrightarrow{\lambda}\in V$, $\overline{x}(\overrightarrow{\lambda}) \in \mathcal{E}^{C_{n-1}}_{C_1 ,\dots, C_{n-2}} \cap U$ is a non-degenerate regular equilibrium state of the integrable system \eqref{systy}.
\end{theorem}
\begin{proof}
Let us define first the smooth function $\mathcal{F}:\Omega\times\mathbb{R}^{n-2}\rightarrow\mathbb{R}^{n} $ given by
$$\mathcal{F}(\overline{x},(\lambda_1,\dots,\lambda_{n-2})):=\nabla C_{n-1}(\overline{x})+\lambda_1 \nabla C_1 (\overline{x})+\dots +\lambda_{n-2} \nabla C_{n-2} (\overline{x}),$$
for every $(\overline{x},(\lambda_1,\dots,\lambda_{n-2}))\in \Omega\times\mathbb{R}^{n-2}$.

As $\overline{x}_e \in \mathcal{E}^{C_{n-1}}_{C_1 ,\dots, C_{n-2}}\subset \Omega$ is a regular equilibrium state of the integrable system \eqref{systy}, there exists $\overrightarrow{\lambda_{e}}:=(\lambda^{e}_1,\dots,\lambda^{e}_{n-2})\in\mathbb{R}^{n-2}$ such that $$\nabla C_{n-1}(\overline{x}_e)+\lambda^{e}_1 \nabla C_1 (\overline{x}_e)+\dots +\lambda^{e}_{n-2} \nabla C_{n-2} (\overline{x}_e)=0,$$ or equivalently, $\mathcal{F}(\overline{x}_e,(\lambda^{e}_1,\dots,\lambda^{e}_{n-2}))=0$. 

Since $\overline{x}_e$ is non-degenerate, it follows that $\det{(\operatorname{Hess}F_{\overrightarrow{\lambda_{e}}}(\overline{x}_e))}\neq 0$, where $F_{\overrightarrow{\lambda_{e}}}:\Omega\rightarrow \mathbb{R}$ is the smooth function given by $F_{\overrightarrow{\lambda_{e}}}:=C_{n-1}+\lambda^{e}_1 C_{1} +\dots +\lambda^{e}_{n-2} C_{n-2}$.

Hence, $\mathrm{D}_{\overline{x}}\mathcal{F}(\overline{x}_e,(\lambda^{e}_1,\dots,\lambda^{e}_{n-2}))=\operatorname{Hess}F_{\overrightarrow{\lambda_{e}}}(\overline{x}_e)$ is invertible, and consequently by the implicit function theorem, there exist $W\subseteq \mathbb{R}^{n-2}$, an open neighborhood of $\overrightarrow{\lambda_{e}}$, $U\subseteq \Omega$, an open neighborhood of $\overline{x}_e$, and a unique smooth function $\overline{x}: W\rightarrow U$ such that $\overline{x}(\overrightarrow{\lambda_{e}})=\overline{x}_e$, and  
$$
\mathcal{F}(\overline{x}(\overrightarrow{\lambda}),\overrightarrow{\lambda})=0, ~(\forall)\overrightarrow{\lambda}=(\lambda_1,\dots,\lambda_{n-2})\in W.
$$
Consequently, we obtain that 
\begin{equation}\label{rely1}
\nabla C_1 (\overline{x}(\overrightarrow{\lambda}))\wedge \dots \wedge \nabla C_{n-1} (\overline{x}(\overrightarrow{\lambda})) =0, ~ (\forall)\overrightarrow{\lambda}\in W,
\end{equation}
and hence $\overline{x}(\overrightarrow{\lambda})\in\mathcal{E}^{X}$, i.e., $\overline{x}(\overrightarrow{\lambda})$ is an equilibrium state of the integrable system \eqref{systy} for every $\overrightarrow{\lambda}\in W$.

Since $\overline{x}_e$ is a non-degenerate regular equilibrium state, it follows that apart from the equilibrium condition $\nabla C_1 (\overline{x}_e)\wedge \dots \wedge \nabla C_{n-1} (\overline{x}_e) = 0,$ $\overline{x}_e$ also verifies the following relations:
\begin{itemize}
\item [(i)] $\nu(\overline{x}_e)\neq 0,$
\item [(ii)] $\nabla C_1 (\overline{x}_e)\wedge \dots \wedge \nabla C_{n-2} (\overline{x}_e) \neq 0,$
\item [(iii)] $\det{(\operatorname{Hess}F_{\overrightarrow{\lambda_{e}}}(\overline{x}_e))}\neq 0.$
\end{itemize}
As $\nu,C_1,\dots,C_{n-1}\in\mathcal{C}^{\infty}(\Omega,\mathbb{R})$ are smooth functions, and $(i),(ii),(iii)$ are open conditions, it follows that there exists an open neighborhood $V\subseteq W$ of $\overrightarrow{\lambda_{e}}$ such that for every $\overrightarrow{\lambda}=(\lambda_1 ,\dots,\lambda_{n-2})\in V$ the following relations hold true:
\begin{itemize}
\item [(i)] $\nu(\overline{x}(\overrightarrow{\lambda}))\neq 0,$
\item [(ii)] $\nabla C_1 (\overline{x}(\overrightarrow{\lambda}))\wedge \dots \wedge \nabla C_{n-2} (\overline{x}(\overrightarrow{\lambda})) \neq 0,$
\item [(iii)] $\det{(\operatorname{Hess}F_{\overrightarrow{\lambda}}(\overline{x}(\overrightarrow{\lambda})))}\neq 0,$ where $F_{\overrightarrow{\lambda}}:=C_{n-1}+\lambda_1 C_{1} +\dots +\lambda_{n-2} C_{n-2}$.
\end{itemize}
Hence, the above relations together with the equality \eqref{rely1} imply that each element which belongs to the image of the smooth function $\overline{x}: V\rightarrow U$, is a non-degenerate regular equilibrium state of the integrable system \eqref{systy}.
\end{proof}

\section{A geometric invariant of non-degenerate equilibria of completely integrable systems}

As the main purpose of this article is to study the stability of non-degenerate regular equilibrium states of the completely integrable system \eqref{sys} (realized as the Hamiltonian system \eqref{systy}), our first step in this direction will be to define a local geometric invariant associated to each non-degenerate regular equilibrium state to be analyzed.

In order to do that, let $\overline{x}_e \in \mathcal{E}^{C_{n-1}}_{C_1 ,\dots, C_{n-2}} \subseteq \Omega$ be a regular equilibrium state of the completely integrable system \eqref{sys}, realized as the Hamiltonian dynamical system \eqref{systy}, i.e.,  $$(\Omega,\{\cdot,\cdot\}_{\nu;C_1,\dots,C_{n-2}},H = C_{n-1}).$$ 

Recall from the previous section that, since $\overline{x}_e \in \mathcal{E}^{C_{n-1}}_{C_1 ,\dots, C_{n-2}}=\{\overline{x} \in \Omega \ | \ \nu (\overline{x})\neq 0, \ \nabla C_1 (\overline{x})\wedge \dots \wedge \nabla C_{n-2} (\overline{x}) \neq 0, \ \nabla C_1 (\overline{x})\wedge \dots \wedge \nabla C_{n-2} (\overline{x})\wedge \nabla H (\overline{x}) =0\}$, there exists $(\lambda^{e}_1, \dots,\lambda^{e}_{n-2})\in\mathbb{R}^{n-2}$, such that $\nabla H(\overline{x}_e)+\lambda^{e}_1 \nabla C_1 (\overline{x}_e)+\dots +\lambda^{e}_{n-2} \nabla C_{n-2} (\overline{x}_e)=0$. Equivalently, the later condition can be written as$$\mathrm{d}(H+\lambda^{e}_1 C_1 +\dots + \lambda^{e}_{n-2} C_{n-2})(\overline{x}_e)=0,$$
since by the definition of the gradient vector field we have that
$$
\langle \nabla (H+\lambda^{e}_1 C_1 +\dots +\lambda^{e}_{n-2} C_{n-2}) (\overline{x}_e),u\rangle=\mathrm{d}(H+\lambda^{e}_1 C_1 +\dots + \lambda^{e}_{n-2} C_{n-2})(\overline{x}_e) \cdot u,
$$
for every $u\in T_{\overline{x}_e}\Omega=T_{\overline{x}_e}\mathbb{R}^{n}\cong \mathbb{R}^{n}$.

In order to have more compact notations, we denote $\overrightarrow{\lambda_{e}}:=(\lambda^{e}_1,\dots,\lambda^{e}_{n-2})\in\mathbb{R}^{n-2}$, and $F_{\overrightarrow{\lambda_{e}}}:\Omega\subseteq \mathbb{R}^{n} \longrightarrow \mathbb{R}$, 
\begin{equation}\label{flamb}
F_{\overrightarrow{\lambda_{e}}}:=H+\lambda^{e}_1 C_1 +\dots + \lambda^{e}_{n-2} C_{n-2}.
\end{equation}

At this stage, we have all necessary ingredients to introduce the main protagonist of this work, i.e., a scalar quantity associated to the regular equilibrium state $\overline{x}_e \in \mathcal{E}^{C_{n-1}}_{C_1 ,\dots, C_{n-2}}$, whose sign will tell us the stability of $\overline{x}_e$. In order to define this scalar quantity, the regular equilibrium point $\overline{x}_e$ needs also to be non-degenerate, in the sense of Definition \eqref{ndgequil}. As will be showed later (see Theorem \eqref{spec1}), despite of the apparent dependence on the Hamiltonian realization of the system \eqref{sys}, this scalar quantity depends only on the vector field $X$ and the associated equilibrium point, $\overline{x}_e$. 

\begin{definition}\label{invi}
For each non-degenerate regular equilibrium point of the system \eqref{systy}, $\overline x_e\in \mathcal{E}^{C_{n-1}}_{C_1 ,\dots, C_{n-2}}$, we define the scalar quantity
\begin{align*}
\mathcal{I}_{X}(\overline{x}_e):=&\nu^2(\overline{x}_e)\cdot \det{(\operatorname{Hess}F_{\overrightarrow{\lambda_{e}}}(\overline{x}_e))}\cdot <[\operatorname{Hess}F_{\overrightarrow{\lambda_{e}}}(\overline{x}_e)]^{-1}\cdot\nabla C_1 (\overline{x}_e)\wedge \\
&\wedge\dots \wedge [\operatorname{Hess}F_{\overrightarrow{\lambda_{e}}}(\overline{x}_e)]^{-1}\cdot\nabla C_{n-2} (\overline{x}_e), \nabla C_1 (\overline{x}_e)\wedge \dots \wedge \nabla C_{n-2} (\overline{x}_e)>_{n-2},
\end{align*}
where $F_{\overrightarrow{\lambda_{e}}}\in\mathcal{C}^{\infty}(\Omega,\mathbb{R})$ stands the smooth function associated to $\overline x_e$, given by the relation \eqref{flamb}.
\end{definition}

Next, we introduce the necessary tools to show the main property of $\mathcal{I}_{X}(\overline{x}_e)$, namely, the local invariance with respect to smooth deformations around $\overline x_e$. In order to do that, let $\Omega^{\prime}\subseteq \Omega$ be an arbitrary open neighborhood of $\overline{x}_e$, and let $\Phi : \Omega^{\prime}\subseteq \mathbb{R}^{n} \longrightarrow W^{\prime}:=\Phi(\Omega^{\prime})\subseteq \mathbb{R}^{n}$ be an arbitrary smooth diffeomorphism. Let us recall now a result from \cite{tgn}, which provides the explicit formula of $\Phi_{\star}X$, the push forward of the vector field $X$ by the diffeomorphism $\Phi$. For the sake of simplicity, we shall use in the sequel the same notation for the vector field $X$, and also for its restriction to $\Omega^{\prime}$.

\begin{theorem}[\cite{tgn}]\label{tsv}
Let 
\begin{equation}\label{hax}
X = \sum_{i=1}^{n}\nu \cdot \dfrac{\partial(C_1,\dots,C_{n-2},x_i,H)}{\partial(x_1,\dots,x_n)}\cdot\dfrac{\partial}{\partial{x_i}},
\end{equation}
be the vector field associated to the completely integrable system \eqref{sys}, written as the Hamiltonian dynamical system $\left(\Omega,\{\cdot,\cdot\}_{\nu;C_1,\dots,C_{n-2}},H=C_{n-1}\right)$. Let $\Omega^{\prime}\subseteq \Omega$ be an open subset, and let $\Phi:\Omega^{\prime}\rightarrow W^{\prime}:=\Phi(\Omega^{\prime})$ be a smooth diffeomorphism. 

Then, $\Phi_{\star}X$ is a Hamiltonian vector field too, with Hamiltonian $\Phi_{\star}H=\Phi_{\star}C_{n-1}$, defined on the Poisson manifold $\left(W^{\prime},\{\cdot,\cdot\}_{\nu_{\Phi};\Phi_{\star}C_1,\dots,\Phi_{\star}C_{n-2}}\right)$, and has the expression
\begin{equation}\label{vpush}
\Phi_{\star}X = \sum_{i=1}^{n}\nu_{\Phi} \cdot \dfrac{\partial(\Phi_{\star}C_1,\dots,\Phi_{\star}C_{n-2},y_i,\Phi_{\star}H)}{\partial(y_1,\dots,y_n)}\cdot\dfrac{\partial}{\partial{y_i}},
\end{equation}
where $\nu_{\Phi}=\Phi_{\star}\nu\cdot \Phi_{\star}\operatorname{Jac}(\Phi)$, and $(y_1,\dots,y_n)=\Phi(x_1,\dots,x_n)$, denote the local coordinates on $W$.
\end{theorem}

Next result describes the relation between non-degenerate regular equilibrium states of the vector field $X$ written in the form \eqref{hax}, and the corresponding equilibrium states of the vector field $\Phi_{\star}X$, where $\Phi : \Omega^{\prime}\subseteq \Omega\subseteq \mathbb{R}^{n} \longrightarrow W^{\prime}:=\Phi(\Omega^{\prime})\subseteq \mathbb{R}^{n}$ is an arbitrary smooth diffeomorphism.

\begin{proposition}\label{lemma1}
Let $\overline{x}_e \in \mathcal{E}^{C_{n-1}}_{C_1 ,\dots, C_{n-2}}$ be a regular equilibrium point of the vector field $X$ written in the form \eqref{hax}. Let $\Omega^{\prime}\subseteq \Omega$ be an open neighborhood of $\overline{x}_e$, and let $\Phi:\Omega^{\prime}\rightarrow W^{\prime}:=\Phi(\Omega^{\prime})$ be a smooth diffeomorphism. Then, $\Phi(\overline{x}_e)\in W^{\prime}$ is a regular equilibrium point of the vector field $\Phi_{\star}X$ written in the Hamiltonian form $$\left(W^{\prime},\{\cdot,\cdot\}_{\nu_{\Phi};\Phi_{\star}C_1,\dots,\Phi_{\star}C_{n-2}},\Phi_{\star}H = \Phi_{\star}C_{n-1}\right),$$ such that 
\begin{align*}
\Phi(\overline{x}_e)\in\mathcal{E}^{\Phi_{\star} C_{n-1}}_{\Phi_{\star} C_1 ,\dots, \Phi_{\star} C_{n-2}}:=&\{\overline{y}\in W^{\prime}\ | \ \nu_{\Phi}(\overline{y})\neq 0, \ \nabla (\Phi_{\star}C_{1}) (\overline{y})\wedge\dots\wedge\nabla (\Phi_{\star}C_{n-2}) (\overline{y})\neq 0,\\
& \ \nabla (\Phi_{\star}C_{1}) (\overline{y})\wedge\dots\wedge\nabla (\Phi_{\star}C_{n-2}) (\overline{y})\wedge\nabla (\Phi_{\star}H) (\overline{y})= 0\}.
\end{align*}

Also, if $F_{\overrightarrow{\lambda_{e}}}:=H+\lambda^{e}_1 C_1 +\dots + \lambda^{e}_{n-2} C_{n-2}$, where $\overrightarrow{\lambda_{e}}:=(\lambda^{e}_1, \dots,\lambda^{e}_{n-2})\in\mathbb{R}^{n-2}$  such that $\mathrm{d}(H+\lambda^{e}_1 C_1 +\dots + \lambda^{e}_{n-2} C_{n-2})(\overline{x}_e)=0$, then $\mathrm{d}(\Phi_{\star}F_{\overrightarrow{\lambda_{e}}})(\Phi(\overline{x}_e))=0$. 

Moreover, if $\overline{x}_e$ is non-degenerate  (i.e., $\det{(\operatorname{Hess}F_{\overrightarrow{\lambda_{e}}}(\overline{x}_e))}\neq 0$), then so is $\Phi(\overline{x}_e)$ (i.e., $\det{(\operatorname{Hess}(\Phi_{\star}F_{\overrightarrow{\lambda_{e}}})(\Phi(\overline{x}_e)))}\neq 0$).
\end{proposition} 
\begin{proof}
Since $\nu(\overline{x}_e)\neq 0$ and $\Phi$ is a diffeomorphism, using the expression of the vector field $\Phi_{\star}X$ given in Theorem \eqref{tsv}, we obtain that 
\begin{equation*}
\nu_{\Phi}(\Phi(\overline{x}_e))=\Phi_{\star}\nu(\Phi(\overline{x}_e))\cdot \Phi_{\star}\operatorname{Jac}(\Phi)(\Phi(\overline{x}_e))=\nu(\overline{x}_e)\cdot\operatorname{Jac}(\Phi)(\overline{x}_e)\neq 0.
\end{equation*}
Next, since $\Phi$ is a diffeomorphism, the linear map $\mathrm{D}\Phi^{-1}(\Phi(\overline{x}_e)):\mathbb{R}^{n}\longrightarrow\mathbb{R}^{n}$ is an isomorphism of vector spaces, as well as its transpose (i.e., the adjoint map with respect to the canonical inner product on $\mathbb{R}^n$). Taking into account that for every smooth function $F\in\mathcal{C}^{\infty}(\Omega^{\prime},\mathbb{R})$, the following formula holds true
\begin{equation}\label{grdcomp}
[\mathrm{D}\Phi^{-1}(\Phi(\overline{x}_e))]^{\top}\cdot \nabla F(\overline{x}_e)=\nabla(\Phi_{\star}F)(\Phi(\overline{x}_e)),
\end{equation}
we get that the linear independence of the vectors $\nabla C_1 (\overline{x}_e), \dots, \nabla C_{n-2} (\overline{x}_e)$, is equivalent to the linear independence of the vectors $\nabla (\Phi_{\star}C_1) (\Phi(\overline{x}_e)), \dots, \nabla (\Phi_{\star}C_{n-2}) (\Phi(\overline{x}_e))$. Equivalently, we have that $\nabla C_1 (\overline{x}_e)\wedge \dots \wedge \nabla C_{n-2} (\overline{x}_e)\neq 0$ if and only if $\nabla (\Phi_{\star}C_1) (\Phi(\overline{x}_e))\wedge \dots \wedge \nabla (\Phi_{\star}C_{n-2}) (\Phi(\overline{x}_e))\neq 0$. 

The same argument implies that $\nabla C_1 (\overline{x}_e)\wedge \dots \wedge \nabla C_{n-2} (\overline{x}_e)\wedge \nabla H(\overline{x}_e) = 0$ if and only if $\nabla (\Phi_{\star}C_1) (\Phi(\overline{x}_e))\wedge \dots \wedge \nabla (\Phi_{\star}C_{n-2}) (\Phi(\overline{x}_e))\wedge\nabla (\Phi_{\star}H) (\Phi(\overline{x}_e))= 0$.

Moreover, since $\overline{x}_e \in \mathcal{E}^{C_{n-1}}_{C_1 ,\dots, C_{n-2}}$, recall that there exists $\overrightarrow{\lambda_{e}}=(\lambda^{e}_1, \dots,\lambda^{e}_{n-2})\in\mathbb{R}^{n-2}$, such that $\nabla H(\overline{x}_e)+\lambda^{e}_1 \nabla C_1 (\overline{x}_e)+\dots +\lambda^{e}_{n-2} \nabla C_{n-2} (\overline{x}_e)=0$ (or equivalently, $\mathrm{d}(H+\lambda^{e}_1 C_1 +\dots + \lambda^{e}_{n-2} C_{n-2})(\overline{x}_e)=\mathrm{d}F_{\overrightarrow{\lambda_{e}}}(\overline{x}_e)=0$). 

Applying the transpose map $[\mathrm{D}\Phi^{-1}(\Phi(\overline{x}_e))]^{\top}$ to the above equality and using the formula \eqref{grdcomp}, we get that $$\nabla (\Phi_{\star}H)(\Phi(\overline{x}_e))+\lambda^{e}_1 \nabla (\Phi_{\star}C_1) (\Phi(\overline{x}_e))+\dots +\lambda^{e}_{n-2} \nabla (\Phi_{\star}C_{n-2}) (\Phi(\overline{x}_e))=0,$$ or equivalently, $\mathrm{d}(\Phi_{\star}H+\lambda^{e}_1 \Phi_{\star}C_1 +\dots + \lambda^{e}_{n-2} \Phi_{\star}C_{n-2})(\Phi(\overline{x}_e))=\mathrm{d}(\Phi_{\star}F_{\overrightarrow{\lambda_{e}}})(\Phi(\overline{x}_e))=0$.

Since, $\overline{x}_e$ is a critical point for $F_{\overrightarrow{\lambda_{e}}}$ (and $\Phi(\overline{x}_e)$ is a critical point of $\Phi_{\star}F_{\overrightarrow{\lambda_{e}}}$), the following formula holds true
\begin{equation}\label{hesscomp}
\operatorname{Hess}(\Phi_{\star}F_{\overrightarrow{\lambda_{e}}})(\Phi(\overline{x}_e))=[\mathrm{D}\Phi^{-1}(\Phi(\overline{x}_e))]^{\top}\circ \operatorname{Hess}F_{\overrightarrow{\lambda_{e}}}(\overline{x}_e)\circ \mathrm{D}\Phi^{-1}(\Phi(\overline{x}_e)),
\end{equation}
and hence we obtain that if $\det{(\operatorname{Hess}F_{\overrightarrow{\lambda_{e}}}(\overline{x}_e))}\neq 0$, then
\begin{equation*}
\det{(\operatorname{Hess}(\Phi_{\star}F_{\overrightarrow{\lambda_{e}}})(\Phi(\overline{x}_e)))}=\left(\operatorname{Jac}(\Phi)(\overline{x}_e)\right)^{-2}\cdot\det{(\operatorname{Hess}F_{\overrightarrow{\lambda_{e}}}(\overline{x}_e))}\neq 0.
\end{equation*}
\end{proof}

Let us state now the main result of this section, and also the main tool we need to prove the stability criterion for the non-degenerate regular equilibrium states of the system \eqref{systy}. 

\begin{theorem}\label{frstmain}
Let $\overline{x}_e \in \mathcal{E}^{C_{n-1}}_{C_1 ,\dots, C_{n-2}}$ be a non-degenerate regular equilibrium point of the vector field $X$ written in the form \eqref{hax}, and let $F_{\overrightarrow{\lambda_{e}}}\in\mathcal{C}^{\infty}(\Omega,\mathbb{R})$ be the associated smooth function given by the relation \eqref{flamb}. Let $\Omega^{\prime}\subseteq \Omega$ be an arbitrary open neighborhood of $\overline{x}_e$, and let $\Phi : \Omega^{\prime}\subseteq \mathbb{R}^{n} \longrightarrow W^{\prime}:=\Phi(\Omega^{\prime})\subseteq \mathbb{R}^{n}$ be a smooth diffeomorphism. Then the following equality holds true:
\begin{equation*}
\mathcal{I}_{\Phi_{\star}X}(\Phi(\overline{x}_e))=\mathcal{I}_{X}(\overline{x}_e).
\end{equation*}
\end{theorem}
\begin{proof}
Recall from Proposition \eqref{lemma1} that $\Phi(\overline{x}_e)\in\mathcal{E}^{\Phi_{\star} C_{n-1}}_{\Phi_{\star} C_1 ,\dots, \Phi_{\star} C_{n-2}}$ is a non-degenerate regular equilibrium point of the vector field $\Phi_{\star}X$ written in the Hamiltonian form $$\left(W^{\prime},\{\cdot,\cdot\}_{\nu_{\Phi};\Phi_{\star}C_1,\dots,\Phi_{\star}C_{n-2}},\Phi_{\star}H = \Phi_{\star}C_{n-1}\right).$$ Moreover, $\Phi(\overline{x}_e)\in\mathcal{E}^{\Phi_{\star} C_{n-1}}_{\Phi_{\star} C_1 ,\dots, \Phi_{\star} C_{n-2}}$ is a non-degenerate critical point of the smooth function $\Phi_{\star}F_{\overrightarrow{\lambda_{e}}}$, i.e., $\det{(\operatorname{Hess}(\Phi_{\star}F_{\overrightarrow{\lambda_{e}}})(\Phi(\overline{x}_e)))}\neq 0$. Using the formulas \eqref{grdcomp}, \eqref{hesscomp}, we obtain successively:
\begin{align*}
\mathcal{I}&_{\Phi_{\star}X}(\Phi(\overline{x}_e))=\nu_{\Phi}^2(\Phi(\overline{x}_e))\cdot \det{(\operatorname{Hess}(\Phi_{\star}F_{\overrightarrow{\lambda_{e}}})(\Phi(\overline{x}_e)))}\cdot <[\operatorname{Hess}(\Phi_{\star}F_{\overrightarrow{\lambda_{e}}})(\Phi(\overline{x}_e))]^{-1}\cdot\\
&\nabla(\Phi_{\star} C_1) (\Phi(\overline{x}_e))\wedge\dots \wedge [\operatorname{Hess}(\Phi_{\star}F_{\overrightarrow{\lambda_{e}}})(\Phi(\overline{x}_e))]^{-1}\cdot\nabla(\Phi_{\star} C_{n-2}) (\Phi(\overline{x}_e)),\\
& \nabla (\Phi_{\star}C_1) (\Phi(\overline{x}_e))\wedge \dots \wedge \nabla(\Phi_{\star} C_{n-2}) (\Phi(\overline{x}_e))>_{n-2}\\
&=[\nu^{2}(\overline{x}_e)\cdot(\operatorname{Jac}(\Phi)(\overline{x}_e))^{2}]\cdot[\left(\operatorname{Jac}(\Phi)(\overline{x}_e)\right)^{-2}\cdot\det{(\operatorname{Hess}F_{\overrightarrow{\lambda_{e}}}(\overline{x}_e))}]\cdot\\
&<[(\mathrm{D}\Phi^{-1}(\Phi(\overline{x}_e)))^{\top}\circ \operatorname{Hess}F_{\overrightarrow{\lambda_{e}}}(\overline{x}_e)\circ \mathrm{D}\Phi^{-1}(\Phi(\overline{x}_e))]^{-1}\cdot[(\mathrm{D}\Phi^{-1}(\Phi(\overline{x}_e)))^{\top}\cdot \nabla C_1(\overline{x}_e)]\\
&\wedge\dots\wedge [(\mathrm{D}\Phi^{-1}(\Phi(\overline{x}_e)))^{\top}\circ \operatorname{Hess}F_{\overrightarrow{\lambda_{e}}}(\overline{x}_e)\circ \mathrm{D}\Phi^{-1}(\Phi(\overline{x}_e))]^{-1}\cdot[(\mathrm{D}\Phi^{-1}(\Phi(\overline{x}_e)))^{\top}\cdot \nabla C_{n-2}(\overline{x}_e)],\\
& (\mathrm{D}\Phi^{-1}(\Phi(\overline{x}_e)))^{\top}\cdot \nabla C_1(\overline{x}_e) \wedge \dots \wedge (\mathrm{D}\Phi^{-1}(\Phi(\overline{x}_e)))^{\top}\cdot \nabla C_{n-2}(\overline{x}_e)>_{n-2}\\
&=\nu^{2}(\overline{x}_e)\cdot\det{(\operatorname{Hess}F_{\overrightarrow{\lambda_{e}}}(\overline{x}_e))}\cdot <\mathrm{D}\Phi(\overline{x}_e)\cdot[[\operatorname{Hess}F_{\overrightarrow{\lambda_{e}}}(\overline{x}_e)]^{-1}\cdot\nabla C_{1}(\overline{x}_e)]\\
&\wedge\dots\wedge \mathrm{D}\Phi(\overline{x}_e)\cdot[[\operatorname{Hess}F_{\overrightarrow{\lambda_{e}}}(\overline{x}_e)]^{-1}\cdot\nabla C_{n-2}(\overline{x}_e)], (\mathrm{D}\Phi^{-1}(\Phi(\overline{x}_e)))^{\top}\cdot \nabla C_1(\overline{x}_e)\\
&\wedge\dots \wedge (\mathrm{D}\Phi^{-1}(\Phi(\overline{x}_e)))^{\top}\cdot \nabla C_{n-2}(\overline{x}_e)>_{n-2}\\
&=\nu^{2}(\overline{x}_e)\cdot\det{(\operatorname{Hess}F_{\overrightarrow{\lambda_{e}}}(\overline{x}_e))}\cdot <[\operatorname{Hess}F_{\overrightarrow{\lambda_{e}}}(\overline{x}_e)]^{-1}\cdot\nabla C_{1}(\overline{x}_e)\\
&\wedge\dots\wedge [\operatorname{Hess}F_{\overrightarrow{\lambda_{e}}}(\overline{x}_e)]^{-1}\cdot\nabla C_{n-2}(\overline{x}_e), (\mathrm{D}\Phi(\overline{x}_e))^T \cdot[(\mathrm{D}\Phi^{-1}(\Phi(\overline{x}_e)))^{\top}\cdot \nabla C_1(\overline{x}_e)]\\
&\wedge\dots \wedge (\mathrm{D}\Phi(\overline{x}_e))^T\cdot[(\mathrm{D}\Phi^{-1}(\Phi(\overline{x}_e)))^{\top}\cdot \nabla C_{n-2}(\overline{x}_e)]>_{n-2}\\
&=\nu^{2}(\overline{x}_e)\cdot\det{(\operatorname{Hess}F_{\overrightarrow{\lambda_{e}}}(\overline{x}_e))}\cdot <[\operatorname{Hess}F_{\overrightarrow{\lambda_{e}}}(\overline{x}_e)]^{-1}\cdot\nabla C_{1}(\overline{x}_e)\\
&\wedge\dots\wedge [\operatorname{Hess}F_{\overrightarrow{\lambda_{e}}}(\overline{x}_e)]^{-1}\cdot\nabla C_{n-2}(\overline{x}_e), \nabla C_1(\overline{x}_e)\wedge\dots\wedge\nabla C_{n-2}(\overline{x}_e)>_{n-2} = \mathcal{I}_{X}(\overline{x}_e),
\end{align*}
and hence we get the conclusion.
\end{proof}

\section{A stability criterion for non-degenerate equilibria of completely integrable systems}

The aim of this section is to present the main result of this article, which provides a criterion to test the stability of non-degenerate regular equilibrium states of the completely integrable system \eqref{sys} written in the Hamiltonian form \eqref{systy}. More precisely, in the notations of the previous section, let $\overline{x}_e$ be a non-degenerate regular equilibrium point of the vector field $X$ written in the form \eqref{hax}, and let $\mathcal{I}_{X}(\overline{x}_e)$ be the associated scalar quantity introduced in Definition \eqref{invi}. Then, the stability criterion states that \textit{if $\mathcal{I}_{X}(\overline{x}_e)<0$ then the equilibrium $\overline{x}_e$ is unstable, whereas if $\mathcal{I}_{X}(\overline{x}_e)>0$ then the equilibrium $\overline{x}_e$ is Lyapunov stable}. Moreover, the characteristic polynomial of the linearization of $X$ at $\overline{x}_e$, $\mathfrak{L}^{X}(\overline{x}_e)$, is given by $p_{\mathfrak{L}^{X}(\overline{x}_e)}(\mu)=(-\mu)^{n-2}\cdot\left( \mu^2 + \mathcal{I}_{X}(\overline{x}_e) \right)$.

In order to prove the above mentioned stability test, we need first to recall some technical details. Let us start by recalling some of the main concepts of the Lyapunov stability of equilibrium states of a general dynamical system. In order to do that, let $\Omega \subseteq \mathbb{R}^n$ be an open set, and let $X\in\mathfrak{X}(\Omega)$ be a smooth vector field. If one denotes by $\{F_{t}^X\}_{t}$ the flow of $X$ (i.e., $\dfrac{\mathrm{d}}{\mathrm{d}t}F_{t}^X(\overline{x})=X(F_{t}^X(\overline{x}))$, $F_{0}^X(\overline{x})=\overline{x}$, for all $\overline{x}\in\Omega$), then a point $\overline{x}_e \in \Omega$ is called an \textit{equilibrium point} of $X$, if $X(\overline{x}_e)=0$, or equivalently, if $F_{t}^X(\overline{x}_e)= \overline{x}_e$ for all $t\in\mathbb{R}$.

An equilibrium point $\overline{x}_e \in \Omega$ of the vector field $X$ is called \textit{Lyapunov stable}, or \textit{nonlinearly stable}, if for every open neighborhood $U\subseteq \Omega$ of $\overline{x}_e$, there exists an open neighborhood $V\subseteq U$ of $\overline{x}_e$ such that $F_{t}^X(\overline{x})\in U$ for any $\overline{x}\in V$ and any $t\geq 0$. An equilibrium state which is not Lyapunov stable is called \textit{unstable}.

Let $\overline{x}_e \in \Omega$ be an equilibrium point of the vector field $X\in\mathfrak{X}(\Omega)$. Recall that the \textit{linearization} of the vector field $X$ at the equilibrium point $\overline{x}_e$, is the linear map $\mathfrak{L}^{X}(\overline{x}_e):\mathbb{R}^{n}\longrightarrow \mathbb{R}^{n}$ defined by
\begin{equation*}
\mathfrak{L}^{X}(\overline{x}_e) \cdot v:=\frac{\mathrm{d}}{\mathrm{d}t}\left( \mathrm{D}F_{t}^X(\overline{x}_e) \cdot v \right)|_{t=0}, \ \text{for all} \ v\in \mathbb{R}^{n}.
\end{equation*}

If the spectrum of the linear map $\mathfrak{L}^{X}(\overline{x}_e)$ lies in the strict left-half complex plane, or on the imaginary axis, the equilibrium point $\overline{x}_e \in \Omega$ is called \textit{spectrally stable}. The equilibrium point $\overline{x}_e$ is called \textit{spectrally unstable} if at least one eigenvalue of $\mathfrak{L}^{X}(\overline{x}_e)$ has strictly positive real part. Since Lyapunov stability implies spectral stability, a spectrally unstable equilibrium point, will also be an unstable equilibrium point (for details see, e.g., \cite{abrahammarsden}).

Let us present now a result which provides the relation between the linearization of the vector field $X$ at an equilibrium point $\overline{x}_e \in \Omega^{\prime}\subseteq\Omega$, and the linearization of the vector field $\Phi_{\star}X$ at the equilibrium point $\Phi(\overline{x}_e) \in W^{\prime}:=\Phi(\Omega^{\prime})\subseteq \mathbb{R}^{n}$, where $\Omega^{\prime}\subseteq \mathbb{R}^{n}$ is an open neighborhood of $\overline{x}_e$, and $\Phi: \Omega^{\prime}\subseteq\mathbb{R}^{n}\longrightarrow W^{\prime}\subseteq\mathbb{R}^{n}$ is a smooth diffeomorphism. 

\begin{proposition}\label{diflin}
Let $\overline{x}_e\in\Omega\subseteq\mathbb{R}^{n}$ be an equilibrium point of the vector field $X$. Let $\Omega^{\prime}\subseteq \Omega$ be an open neighborhood of $\overline{x}_e$, and let $\Phi: \Omega^{\prime}\subseteq\mathbb{R}^{n}\longrightarrow W^{\prime}:=\Phi(\Omega^{\prime})\subseteq\mathbb{R}^{n}$ be a smooth diffeomorphism. Then $\Phi(\overline{x}_e)\in W^{\prime}$ is an equilibrium point of the vector field $\Phi_{\star}X$, and moreover, the following relation between the corresponding linearizations holds true:
\begin{equation}\label{dln}
\mathfrak{L}^{\Phi_{\star}X}(\Phi(\overline{x}_e))=\mathrm{D}\Phi(\overline{x}_e)\circ\mathfrak{L}^{X}(\overline{x}_e)\circ(\mathrm{D}\Phi\left(\overline{x}_e)\right)^{-1}.
\end{equation}
\end{proposition}
\begin{proof}
If one denotes by $\{F_{t}^X\}_{t}$ the flow of $X$, then the flow of $\Phi_{\star}X$ is given by $\{\Phi\circ F_{t}^X \circ \Phi^{-1}\}_{t}$. Hence,  $(\Phi\circ F_{t}^X \circ \Phi^{-1})(\Phi(\overline{x}_e))=\Phi(F_{t}^X (\overline{x}_e))=\Phi(\overline{x}_e)$, for all $t\in\mathbb{R}$ (since $\overline{x}_e$ is an equilibrium point of $X$, i.e., $F_{t}^X (\overline{x}_e)=\overline{x}_e$, for all $t\in\mathbb{R}$), and consequently $\Phi(\overline{x}_e)$ is an equilibrium point of $\Phi_{\star}X$. In order to prove the second statement, let $w\in\mathbb{R}^{n}$ be arbitrary chosen. Then we obtain successively the following equalities:
\begin{align*}
\mathfrak{L}^{\Phi_{\star}X}(\Phi(\overline{x}_e))\cdot w &=\dfrac{\mathrm{d}}{\mathrm{d}t}\left[\mathrm{D}F^{\Phi_{\star}X}_{t}(\Phi(\overline{x}_e))\cdot w\right]|_{t=0}=\dfrac{\mathrm{d}}{\mathrm{d}t}\left[\mathrm{D}(\Phi\circ F^{X}_{t}\circ \Phi^{-1})(\Phi(\overline{x}_e))\cdot w\right]|_{t=0}\\
&=\dfrac{\mathrm{d}}{\mathrm{d}t}\left[\mathrm{D}\Phi(F^{X}_{t}(\overline{x}_e))\cdot (\mathrm{D}F^{X}_{t}(\overline{x}_e)\cdot (\mathrm{D}\Phi^{-1}(\Phi(\overline{x}_e))\cdot w ) ) \right]|_{t=0}\\
&=\dfrac{\mathrm{d}}{\mathrm{d}t}\left[\mathrm{D}\Phi(\overline{x}_e)\cdot (\mathrm{D}F^{X}_{t}(\overline{x}_e)\cdot (\mathrm{D}\Phi^{-1}(\Phi(\overline{x}_e))\cdot w ) ) \right]|_{t=0}\\
&=\mathrm{D}\Phi(\overline{x}_e)\cdot \dfrac{\mathrm{d}}{\mathrm{d}t} \left[ \mathrm{D}F^{X}_{t}(\overline{x}_e)\cdot (\mathrm{D}\Phi^{-1}(\Phi ( \overline{x}_e))\cdot w) \right]|_{t=0}\\
&=\mathrm{D}\Phi(\overline{x}_e)\cdot(\mathfrak{L}^{X}(\overline{x}_e)\cdot(\mathrm{D}\Phi^{-1}(\Phi ( \overline{x}_e))\cdot w))\\
&=(\mathrm{D}\Phi(\overline{x}_e)\circ\mathfrak{L}^{X}(\overline{x}_e)\circ \mathrm{D}\Phi^{-1}(\Phi ( \overline{x}_e)))\cdot w,
\end{align*}
and hence we get the conclusion.
\end{proof}
\begin{corollary}\label{spectrum}
If $\overline{x}_e\in\Omega\subseteq\mathbb{R}^{n}$ is an equilibrium point of the vector field $X$, $\Omega^{\prime}\subseteq \Omega$ is an open neighborhood of $\overline{x}_e$, and $\Phi: \Omega^{\prime}\subseteq\mathbb{R}^{n}\longrightarrow W^{\prime}:=\Phi(\Omega^{\prime})\subseteq\mathbb{R}^{n}$ is a smooth diffeomorphism, then the linear maps $\mathfrak{L}^{\Phi_{\star}X}(\Phi(\overline{x}_e))$ and $\mathfrak{L}^{X}(\overline{x}_e)$ have the same characteristic polynomial.
\end{corollary}

Let us return now to the main problem of this article, namely the stability analysis of the equilibrium states of the completely integrable system \eqref{sys}. In order to do that, let $\overline{x}_e \in \mathcal{E}^{C_{n-1}}_{C_1 ,\dots, C_{n-2}}$ be a regular equilibrium state of the completely integrable system \eqref{sys} modeled as the Hamiltonian dynamical system $(\Omega,\{\cdot,\cdot\}_{\nu;C_1,\dots,C_{n-2}},H = C_{n-1})$ \eqref{systy}. As $\overline{x}_e \in \mathcal{E}^{C_{n-1}}_{C_1 ,\dots, C_{n-2}}=\{\overline{x} \in \Omega \ | \ \nu (\overline{x})\neq 0, \ \nabla C_1 (\overline{x})\wedge \dots \wedge \nabla C_{n-2} (\overline{x}) \neq 0, \ \nabla C_1 (\overline{x})\wedge \dots \wedge \nabla C_{n-1} (\overline{x}) =0\}$, due to the continuity of $\nu,\nabla C_1,\dots,\nabla C_{n-2}$, there exists $\Omega^{\prime}\subseteq\Omega$, an open neighborhood of $\overline{x}_e$, such that $\nu (\overline{x})\neq 0$, and $\nabla C_1 (\overline{x})\wedge \dots \wedge \nabla C_{n-2} (\overline{x}) \neq 0$, for every $\overline{x}\in\Omega^{\prime}$. Note that since $\nu$ is continuous and has no zeros in $\Omega^{\prime}$, then $\nu$ must have constant sign on $\Omega^{\prime}$. Consequently, the Hamiltonian system \eqref{systy} can be brought to Darboux normal form relative to the open set $\Omega^{\prime}$. For the sake of completeness, let us recall the Darboux normal form of completely integrable systems, as stated in \cite{tgn}. For another presentation of the Darboux normal form, see \cite{berme1}.

\begin{theorem}[\cite{tgn}]\label{darbnf}
Let $$X = \sum_{i=1}^{n}\nu \cdot \dfrac{\partial(C_1,\dots,C_{n-2},x_i,H)}{\partial(x_1,\dots,x_n)}\cdot\dfrac{\partial}{\partial{x_i}},$$ be the vector field associated to the completely integrable system \eqref{sys}, written as the Hamiltonian system \eqref{systy}, with Hamiltonian $H:=C_{n-1}$, defined on the Poisson manifold $\left(\Omega,\{\cdot,\cdot\}_{\nu;C_1,\dots,C_{n-2}}\right)$. Let $\overline{x}_e \in\mathcal{E}^{C_{n-1}}_{C_1 ,\dots, C_{n-2}}$ be a regular equilibrium point of the system \eqref{systy}. Let $(\Omega^{\prime},\Phi_1,\Phi_2)$ be a triple consisting of an open neighborhood $\Omega^{\prime}\subseteq \Omega$ of $\overline{x}_e$, and two smooth functions $\Phi_1,\Phi_2 \in\mathcal{C}^{\infty}(\Omega^{\prime},\mathbb{R})$, such that the map $\Phi:\Omega^{\prime}\rightarrow W^{\prime}=\Phi(\Omega^{\prime})$, given by $\Phi=(\Phi_1,\Phi_2,C_1,\dots,C_{n-2})$, is a smooth diffeomorphism, and $\nu(\overline{x})\neq 0$, for every $\overline{x}\in\Omega^{\prime}$. Then $\Phi_{\star}X$, the push forward of the vector field $X$ by $\Phi$, is a Hamiltonian vector field, with Hamiltonian $\Phi_{\star}H=\Phi_{\star}C_{n-1}$, defined on the Poisson manifold $\left(W^{\prime},\{\cdot,\cdot\}_{\nu_{\Phi};\Phi_{\star}C_1,\dots,\Phi_{\star}C_{n-2}}\right)$, and has the expression
\begin{equation}\label{darb}
\Phi_{\star}X = \nu_{\Phi}\cdot\left[ \dfrac{\partial(\Phi_{\star}H)}{\partial y_2}\cdot\dfrac{\partial}{\partial y_1}-\dfrac{\partial(\Phi_{\star}H)}{\partial y_1}\cdot\dfrac{\partial}{\partial y_2}\right],
\end{equation}
where $\nu_{\Phi}=\Phi_{\star}\nu\cdot \Phi_{\star}\operatorname{Jac}(\Phi)$, and $(y_1,\dots,y_n)=\Phi(x_1,\dots,x_n)$, denote the local coordinates on  $W^{\prime}$.
\end{theorem}

According to Corollary \eqref{spectrum}, in order to compute the spectrum of the linearization $\mathfrak{L}^{X}(\overline{x}_e)$, one can compute instead the spectrum of the linearization $\mathfrak{L}^{\Phi_{\star}X}(\Phi(\overline{x}_e))$, where $\Phi$ is the diffeomorphism introduced in Theorem \eqref{darbnf}. More precisely, $\Phi$ is defined by a triple  $(\Omega^{\prime},\Phi_1,\Phi_2)$ consisting of an open neighborhood $\Omega^{\prime}\subseteq \Omega$ of the equilibrium point $\overline{x}_e$, and two smooth functions $\Phi_1,\Phi_2 \in\mathcal{C}^{\infty}(\Omega^{\prime},\mathbb{R})$, such that $\Phi:\Omega^{\prime}\rightarrow W^{\prime}=\Phi(\Omega^{\prime})$, given by $\Phi:=(\Phi_1,\Phi_2,C_1,\dots,C_{n-2})$, is a smooth diffeomorphism. Moreover, recall also from Theorem \eqref{darbnf} that $\nu(\overline{x})\neq 0$, for every $\overline{x}\in\Omega^{\prime}$. Next we present an instability result, based on the fact that if the spectrum of the linearization $\mathfrak{L}^{X}(\overline{x}_e)$ contains an eigenvalue with strictly positive real part, then the equilibrium state $\overline{x}_e$ is unstable.

\begin{theorem}\label{spec1}
Let $\overline{x}_e \in \mathcal{E}^{C_{n-1}}_{C_1 ,\dots, C_{n-2}}$ be a non-degenerate regular equilibrium point of the vector field $X$ written in the form \eqref{hax}, and let $\mathcal{I}_{X}(\overline{x}_e)$ be the associated scalar quantity introduced in Definition \eqref{invi}. Then the characteristic polynomial of the linearization  $\mathfrak{L}^{X}(\overline{x}_e)$ is given by 
$$
p_{\mathfrak{L}^{X}(\overline{x}_e)}(\mu)=(-\mu)^{n-2}\cdot\left( \mu^2 + \mathcal{I}_{X}(\overline{x}_e) \right).
$$
Moreover, if $\mathcal{I}_{X}(\overline{x}_e)<0$, then the equilibrium state $\overline{x}_e$ is unstable.
\end{theorem}
\begin{proof}
Recall from Corollary \eqref{spectrum} that the characteristic polynomials of $\mathfrak{L}^{X}(\overline{x}_e)$ and  $\mathfrak{L}^{\Phi_{\star}X}(\Phi(\overline{x}_e))$ are equal for every diffeomorphism $\Phi$ defined on some open neighborhood of $\overline{x}_e$. Hence, in order to simplify computations, we choose $\Phi$ as given in Theorem \eqref{darbnf} and we compute the associated linearization, $\mathfrak{L}^{\Phi_{\star}X}(\Phi(\overline{x}_e))$. More precisely, we choose a triple $(\Omega^{\prime},\Phi_1,\Phi_2)$ consisting of an open neighborhood $\Omega^{\prime}\subseteq \Omega$ of $\overline{x}_e$, and two smooth functions $\Phi_1,\Phi_2 \in\mathcal{C}^{\infty}(\Omega^{\prime},\mathbb{R})$, such that $\Phi:\Omega^{\prime}\rightarrow W^{\prime}:=\Phi(\Omega^{\prime})$, given by $\Phi:=(\Phi_1,\Phi_2,C_1,\dots,C_{n-2})$, is a smooth diffeomorphism, and moreover $\nu(\overline{x})\neq 0$, for every $\overline{x}\in\Omega^{\prime}$.

In order to compute $\mathfrak{L}^{\Phi_{\star}X}(\Phi(\overline{x}_e))$, let us recall first from Proposition \eqref{lemma1} that $\Phi(\overline{x}_e)\in \mathcal{E}^{\Phi_{\star} C_{n-1}}_{\Phi_{\star} C_1 ,\dots, \Phi_{\star} C_{n-2}}$ is a non-degenerate regular equilibrium point of the vector field $\Phi_{\star}X$ written in the Hamiltonian form $\left(W^{\prime},\{\cdot,\cdot\}_{\nu_{\Phi};\Phi_{\star}C_1,\dots,\Phi_{\star}C_{n-2}},\Phi_{\star}H = \Phi_{\star}C_{n-1}\right)$. Consequently, we have that $\mathrm{d}(\Phi_{\star}F_{\overrightarrow{\lambda_{e}}})(\Phi(\overline{x}_e))=0$, and  $\det{(\operatorname{Hess}(\Phi_{\star}F_{\overrightarrow{\lambda_{e}}})(\Phi(\overline{x}_e)))}\neq 0$, where $F_{\overrightarrow{\lambda_{e}}}:=H+\lambda^{e}_1 C_1 +\dots + \lambda^{e}_{n-2} C_{n-2}$, with $\overrightarrow{\lambda_{e}}:=(\lambda^{e}_1, \dots,\lambda^{e}_{n-2})\in\mathbb{R}^{n-2}$, such that $\mathrm{d}(H+\lambda^{e}_1 C_1 +\dots + \lambda^{e}_{n-2} C_{n-2})(\overline{x}_e)=0$.

Using the definition of $\Phi:\Omega^{\prime}\rightarrow W^{\prime}=\Phi(\Omega^{\prime})$, i.e., $\Phi =(\Phi_1,\Phi_2,C_1,\dots,C_{n-2})$, the expression of $\Phi_{\star}F_{\overrightarrow{\lambda_{e}}}\in\mathcal{C}^{\infty}(W^{\prime},\mathbb{R})$ becomes
\begin{align*}
(\Phi_{\star}F_{\overrightarrow{\lambda_{e}}})(y_1,\dots, y_n)&=(\Phi_{\star}H)(y_1,\dots, y_n)+\lambda^{e}_1 (\Phi_{\star}C_1)(y_1,\dots, y_n) \\
&+\dots + \lambda^{e}_{n-2} (\Phi_{\star}C_{n-2})(y_1,\dots, y_n)\\
&=(\Phi_{\star}H)(y_1,\dots, y_n)+\lambda^{e}_1 y_3 +\dots + \lambda^{e}_{n-2} y_n,
\end{align*}
for every $(y_1,\dots, y_n)\in W^{\prime}$.

Hence, since $\mathrm{d}(\Phi_{\star}F_{\overrightarrow{\lambda_{e}}})(\Phi(\overline{x}_e))=0$, we get that
\begin{align}\label{relimp1}
\begin{split}
\dfrac{\partial(\Phi_{\star}H)}{\partial y_1}(\Phi(\overline{x}_e))&=\dfrac{\partial(\Phi_{\star}H)}{\partial y_2}(\Phi(\overline{x}_e))=0,\\
\dfrac{\partial(\Phi_{\star}H)}{\partial y_{i+2}}(\Phi(\overline{x}_e))&=-\lambda^{e}_i,\ \text{for} \ i\in\{1,\dots,n-2\}.
\end{split}
\end{align}

Moreover, for every $(y_1,\dots, y_n)\in W^{\prime}$, and respectively for each $i,j\in\{1,\dots,n\}$, we have that
\begin{align}\label{relimp2}
\begin{split}
\dfrac{\partial^{2}(\Phi_{\star}F_{\overrightarrow{\lambda_{e}}})}{\partial y_i \partial y_j}(y_1,\dots, y_n)=\dfrac{\partial^{2}(\Phi_{\star}H)}{\partial y_i \partial y_j}(y_1,\dots, y_n).
\end{split}
\end{align}

Since $\dfrac{\partial(\Phi_{\star}H)}{\partial y_1}(\Phi(\overline{x}_e))=\dfrac{\partial(\Phi_{\star}H)}{\partial y_2}(\Phi(\overline{x}_e))=0$, using the expression \eqref{darb} of the vector field $\Phi_{\star}X$, we get the following matrix representation of the linear map $\mathfrak{L}^{\Phi_{\star}X}(\Phi(\overline{x}_e))$ with respect to the canonical basis of $\mathbb{R}^n$ (where $T_{\Phi(\overline{x}_e)}W^{\prime}=T_{\Phi(\overline{x}_e)}\mathbb{R}^{n}\cong\mathbb{R}^n $, $e_i := (0,\dots, 0,1,0,\dots,0)\in\mathbb{R}^{n}\cong\dfrac{\partial}{\partial y_i}|_{\Phi(\overline{x}_e)}\in T_{\Phi(\overline{x}_e)}W^{\prime}=T_{\Phi(\overline{x}_e)}\mathbb{R}^{n}$, for every $i\in\{1,\dots,n\}$):

\begin{equation}\label{lin1}
\mathfrak{L}^{\Phi_{\star}X}(\Phi(\overline{x}_e))=\nu_{\Phi}(\Phi(\overline{x}_e))\cdot
\left[
\begin{array}{c|c}
L_{2,2}(\Phi(\overline{x}_e)) & L_{2,n-2} (\Phi(\overline{x}_e))\\ 
\hline
O_{n-2,2} &  O_{n-2}
\end{array}\right],
\end{equation}
where the blocks $L_{2,2}(\Phi(\overline{x}_e))\in\mathcal{M}_{2}(\mathbb{R})$, $L_{2,n-2} (\Phi(\overline{x}_e))\in\mathcal{M}_{2,n-2}(\mathbb{R})$, are given by
\begin{align*}
L_{2,2}(\Phi(\overline{x}_e)):=
\begin{bmatrix}
\dfrac{\partial^2 (\Phi_{\star}H)}{\partial y_2 \partial y_1}(\Phi(\overline{x}_e)) & \dfrac{\partial^2 (\Phi_{\star}H)}{\partial y_2 ^2}(\Phi(\overline{x}_e))\\
- \dfrac{\partial^2 (\Phi_{\star}H)}{\partial y_1^2}(\Phi(\overline{x}_e)) & - \dfrac{\partial^2 (\Phi_{\star}H)}{\partial y_1 \partial y_2}(\Phi(\overline{x}_e))\\
\end{bmatrix},
\end{align*}
and
\begin{align*}
L_{2,n-2}(\Phi(\overline{x}_e)):=
\begin{bmatrix}
\dfrac{\partial^2 (\Phi_{\star}H)}{\partial y_2 \partial y_3}(\Phi(\overline{x}_e)) & \dots & \dfrac{\partial^2 (\Phi_{\star}H)}{\partial y_2 \partial y_n}(\Phi(\overline{x}_e)) \\
- \dfrac{\partial^2 (\Phi_{\star}H)}{\partial y_1 \partial y_3}(\Phi(\overline{x}_e)) & \dots & - \dfrac{\partial^2 (\Phi_{\star}H)}{\partial y_1 \partial y_n}(\Phi(\overline{x}_e)) \\
\end{bmatrix}.
\end{align*}
Using the formula for the determinant of block matrices, we obtain the following expression for the characteristic polynomial of $\mathfrak{L}^{\Phi_{\star}X}(\Phi(\overline{x}_e))$:
\begin{align}\label{charpol1}
\begin{split}
p_{\mathfrak{L}^{\Phi_{\star}X}(\Phi(\overline{x}_e))}(\mu)&=\det(\mathfrak{L}^{\Phi_{\star}X}(\Phi(\overline{x}_e))-\mu I_{n})\\
&=\det(O_{n-2}-\mu I_{n-2})\cdot\det[\nu_{\Phi}(\Phi(\overline{x}_e)) \cdot L_{2,2}(\overline{x}_e)-\mu I_2]\\
&=(-\mu)^{n-2}\cdot\left[\mu^2 +\nu^2_{\Phi}(\Phi(\overline{x}_e))\cdot\left( \dfrac{\partial^2 (\Phi_{\star}H)}{\partial y_1^2}(\Phi(\overline{x}_e)) \cdot\dfrac{\partial^2 (\Phi_{\star}H)}{\partial y_2^2}(\Phi(\overline{x}_e))\right.\right.\\
&-\left.\left.\left( \dfrac{\partial^2 (\Phi_{\star}H)}{\partial y_1 \partial y_2}(\Phi(\overline{x}_e))\right)^2\right)\right].
\end{split}
\end{align}
Next, we will show that 
\begin{equation}\label{ifeq}
\nu^2_{\Phi}(\Phi(\overline{x}_e))\cdot\left( \dfrac{\partial^2 (\Phi_{\star}H)}{\partial y_1^2}(\Phi(\overline{x}_e)) \cdot\dfrac{\partial^2 (\Phi_{\star}H)}{\partial y_2^2}(\Phi(\overline{x}_e))-\left( \dfrac{\partial^2 (\Phi_{\star}H)}{\partial y_1 \partial y_2}(\Phi(\overline{x}_e))\right)^2\right)=\mathcal{I}_{\Phi_{\star}X}(\Phi(\overline{x}_e)).
\end{equation}
In order to do that, let us recall first the definition of $\mathcal{I}_{\Phi_{\star}X}(\Phi(\overline{x}_e))$, i.e., 
\begin{align}\label{ifi}
\begin{split}
\mathcal{I}&_{\Phi_{\star}X}(\Phi(\overline{x}_e))=\nu_{\Phi}^2(\Phi(\overline{x}_e))\cdot \det{(\operatorname{Hess}(\Phi_{\star}F_{\overrightarrow{\lambda_{e}}})(\Phi(\overline{x}_e)))}\cdot <[\operatorname{Hess}(\Phi_{\star}F_{\overrightarrow{\lambda_{e}}})(\Phi(\overline{x}_e))]^{-1}\cdot\\
&\nabla(\Phi_{\star} C_1) (\Phi(\overline{x}_e))\wedge\dots \wedge [\operatorname{Hess}(\Phi_{\star}F_{\overrightarrow{\lambda_{e}}})(\Phi(\overline{x}_e))]^{-1}\cdot\nabla(\Phi_{\star} C_{n-2}) (\Phi(\overline{x}_e)),\\
& \nabla (\Phi_{\star}C_1) (\Phi(\overline{x}_e))\wedge \dots \wedge \nabla(\Phi_{\star} C_{n-2}) (\Phi(\overline{x}_e))>_{n-2}.
\end{split}
\end{align}
Using the relations \eqref{relimp2}, we obtain the following block matrix representation of the linear map $\operatorname{Hess}(\Phi_{\star}F_{\overrightarrow{\lambda_{e}}})(\Phi(\overline{x}_e))$ with respect to the canonical basis of $\mathbb{R}^n$:
\begin{equation}\label{hess1}
\operatorname{Hess}(\Phi_{\star}F_{\overrightarrow{\lambda_{e}}})(\Phi(\overline{x}_e))=
\left[
\begin{array}{c|c}
H_{2,2}(\Phi(\overline{x}_e)) & H_{2,n-2}(\Phi(\overline{x}_e))\\ 
\hline
H_{n-2,2}(\Phi(\overline{x}_e)) &  H_{n-2,n-2}(\Phi(\overline{x}_e))
\end{array}\right],
\end{equation}
where the blocks $H_{k,l}(\Phi(\overline{x}_e))\in\mathcal{M}_{k,l}(\mathbb{R})$, $k,l\in\{2,n-2\}$, are given by
\begin{align*}
H_{2,2}(\Phi(\overline{x}_e)):&=\left[ \dfrac{\partial^2 (\Phi_{\star}H)}{\partial y_i \partial y_j}(\Phi(\overline{x}_e))\right]_{1\leq i,j\leq 2},  H_{2,n-2}(\Phi(\overline{x}_e)):=\left[ \dfrac{\partial^2 (\Phi_{\star}H)}{\partial y_i \partial y_j}(\Phi(\overline{x}_e)) \right]_{1\leq i\leq 2, 3\leq j\leq n}\\
H_{n-2,2}(\Phi(\overline{x}_e)):&=\left[ \dfrac{\partial^2 (\Phi_{\star}H)}{\partial y_i \partial y_j}(\Phi(\overline{x}_e))\right]_{3\leq i\leq n, 1\leq j\leq 2},  H_{n-2,n-2}(\Phi(\overline{x}_e)):=\left[ \dfrac{\partial^2 (\Phi_{\star}H)}{\partial y_i \partial y_j}(\Phi(\overline{x}_e)) \right]_{3\leq i,j\leq n}.
\end{align*}
Let us denote
\begin{equation}\label{hess2}
\left[\operatorname{Hess}(\Phi_{\star}F_{\overrightarrow{\lambda_{e}}})(\Phi(\overline{x}_e))\right]^{-1}:=
\left[
\begin{array}{c|c}
\widetilde{H}_{2,2}(\Phi (\overline{x}_e)) & \widetilde{H}_{2,n-2}(\Phi( \overline{x}_e))\\ 
\hline
\widetilde{H}_{n-2,2}(\Phi( \overline{x}_e)) &  \widetilde{H}_{n-2,n-2}(\Phi( \overline{x}_e))
\end{array}\right],
\end{equation}
the matrix representation of the inverse linear map, $\left[\operatorname{Hess}(\Phi_{\star}F_{\overrightarrow{\lambda_{e}}})(\Phi(\overline{x}_e))\right]^{-1}$,  with respect to the canonical basis of $\mathbb{R}^n$, where $\widetilde{H}_{k,l}(\Phi(\overline{x}_e))\in\mathcal{M}_{k,l}(\mathbb{R})$, $k,l\in\{2,n-2\}$, stand for the corresponding decomposition blocks. Consequently, the following relations hold true:
\begin{align}\label{invbloc}
\begin{split}
&H_{2,2}(\Phi(\overline{x}_e))\cdot\widetilde{H}_{2,n-2}(\Phi(\overline{x}_e))+ H_{2,n-2}(\Phi(\overline{x}_e))\cdot\widetilde{H}_{n-2,n-2}(\Phi(\overline{x}_e)) = O_{2,n-2},\\
&H_{n-2,2}(\Phi(\overline{x}_e))\cdot\widetilde{H}_{2,n-2}(\Phi(\overline{x}_e))+ H_{n-2,n-2}(\Phi(\overline{x}_e))\cdot\widetilde{H}_{n-2,n-2}(\Phi(\overline{x}_e)) = I_{n-2}.
\end{split}
\end{align}
Using the matrix representation \eqref{hess2}, and taking into account that 
\begin{equation}\label{casgrad}
\nabla(\Phi_{\star}C_i)(\Phi(\overline{x}_e))=\dfrac{\partial}{\partial y_{i+2}}|_{\Phi(\overline{x}_e)}\in T_{\Phi(\overline{x}_e)}W^{\prime}=T_{\Phi(\overline{x}_e)}\mathbb{R}^{n}\cong e_{i+2}\in\mathbb{R}^{n},
\end{equation}
for each $i\in\{1,\dots, n-2\}$, we obtain the following expression for the third factor of the r.h.s. of the formula \eqref{ifi}:
\begin{align}\label{innerok}
\begin{split}
&<[\operatorname{Hess}(\Phi_{\star}F_{\overrightarrow{\lambda_{e}}})(\Phi(\overline{x}_e))]^{-1}\cdot\nabla(\Phi_{\star} C_1) (\Phi(\overline{x}_e))\wedge\dots \wedge [\operatorname{Hess}(\Phi_{\star}F_{\overrightarrow{\lambda_{e}}})(\Phi(\overline{x}_e))]^{-1}\\
&\cdot\nabla(\Phi_{\star} C_{n-2}) (\Phi(\overline{x}_e)), \nabla (\Phi_{\star}C_1) (\Phi(\overline{x}_e))\wedge \dots \wedge \nabla(\Phi_{\star} C_{n-2}) (\Phi(\overline{x}_e))>_{n-2}\\
&=\det\left(\left[ \langle [\operatorname{Hess}(\Phi_{\star}F_{\overrightarrow{\lambda_{e}}})(\Phi(\overline{x}_e))]^{-1}\cdot\nabla(\Phi_{\star} C_i) (\Phi(\overline{x}_e)),\nabla(\Phi_{\star} C_{j}) (\Phi(\overline{x}_e)) \rangle \right]_{1\leq i,j\leq n-2}\right)\\
&=\det (\widetilde{H}_{n-2,n-2}(\Phi (\overline{x}_e)))= \det \left(
\left[
\begin{array}{c|c}
I_2 & \widetilde{H}_{2,n-2}(\Phi(\overline{x}_e))\\ 
\hline
O_{n-2,2} &  \widetilde{H}_{n-2,n-2}(\Phi(\overline{x}_e))
\end{array}\right]
\right).
\end{split}
\end{align}
Hence, using the relations \eqref{hess1}, \eqref{innerok}, and \eqref{invbloc}, the formula \eqref{ifi} becomes
\begin{align*}
\begin{split}
&\mathcal{I}_{\Phi_{\star}X}(\Phi(\overline{x}_e))=\nu_{\Phi}^2(\Phi(\overline{x}_e))\cdot \det\left(\left[
\begin{array}{c|c}
H_{2,2}(\Phi(\overline{x}_e)) & H_{2,n-2}(\Phi(\overline{x}_e))\\ 
\hline
H_{n-2,2}(\Phi(\overline{x}_e)) &  H_{n-2,n-2}(\Phi(\overline{x}_e))
\end{array}\right] \right)\\
& \cdot\det \left( \left[
\begin{array}{c|c}
I_2 & \widetilde{H}_{2,n-2}(\Phi(\overline{x}_e))\\ 
\hline
O_{n-2,2} &  \widetilde{H}_{n-2,n-2}(\Phi(\overline{x}_e))
\end{array}\right]\right)\\
&=\nu_{\Phi}^2(\Phi(\overline{x}_e))\\
&\cdot\det\left( \left[
\begin{array}{c|c}
H_{2,2}(\Phi(\overline{x}_e)) & H_{2,2}(\Phi(\overline{x}_e))\cdot\widetilde{H}_{2,n-2}(\Phi(\overline{x}_e))+ H_{2,n-2}(\Phi(\overline{x}_e))\cdot\widetilde{H}_{n-2,n-2}(\Phi(\overline{x}_e))\\ 
\hline
H_{n-2,2}(\Phi(\overline{x}_e)) & H_{n-2,2}(\Phi(\overline{x}_e))\cdot\widetilde{H}_{2,n-2}(\Phi(\overline{x}_e))+ H_{n-2,n-2}(\Phi(\overline{x}_e))\cdot\widetilde{H}_{n-2,n-2}(\Phi(\overline{x}_e))
\end{array}\right]\right)\\
&=\nu_{\Phi}^2(\Phi(\overline{x}_e))\cdot\det\left( \left[
\begin{array}{c|c}
H_{2,2}(\Phi(\overline{x}_e)) & O_{2,n-2}\\ 
\hline
H_{n-2,2}(\Phi(\overline{x}_e)) & I_{n-2}
\end{array}\right]\right) =\nu_{\Phi}^2(\Phi(\overline{x}_e))\cdot\det\left( H_{2,2}(\Phi(\overline{x}_e))\right)\\
&=\nu^2_{\Phi}(\Phi(\overline{x}_e))\cdot\left( \dfrac{\partial^2 (\Phi_{\star}H)}{\partial y_1^2}(\Phi(\overline{x}_e)) \cdot\dfrac{\partial^2 (\Phi_{\star}H)}{\partial y_2^2}(\Phi(\overline{x}_e))-\left( \dfrac{\partial^2 (\Phi_{\star}H)}{\partial y_1 \partial y_2}(\Phi(\overline{x}_e))\right)^2\right),
\end{split}
\end{align*}
which is exactly the formula \eqref{ifeq}.

Consequently, the expression \eqref{charpol1} of the characteristic polynomial of $\mathfrak{L}^{\Phi_{\star}X}(\Phi(\overline{x}_e))$ becomes
\begin{align*}
\begin{split}
p_{\mathfrak{L}^{\Phi_{\star}X}(\Phi(\overline{x}_e))}(\mu)&=(-\mu)^{n-2}\cdot\left( \mu^2 + \mathcal{I}_{\Phi_{\star}X}(\Phi(\overline{x}_e)) \right).
\end{split}
\end{align*}

As the characteristic polynomials of $\mathfrak{L}^{X}(\overline{x}_e)$ and $\mathfrak{L}^{\Phi_{\star}X}(\Phi(\overline{x}_e))$ are equal (via Corollary \eqref{spectrum}), and $\mathcal{I}_{X}(\overline{x}_e)=\mathcal{I}_{\Phi_{\star}X}(\Phi(\overline{x}_e))$ (via Theorem \eqref{frstmain}), we obtain that
\begin{align*}
p_{\mathfrak{L}^{X}(\overline{x}_e)}(\mu)&=(-\mu)^{n-2}\cdot\left( \mu^2 + \mathcal{I}_{X}(\overline{x}_e) \right).
\end{align*}
Consequently, if $\mathcal{I}_{X}(\overline{x}_e)<0$ then the linearization $\mathfrak{L}^{X}(\overline{x}_e)$ contains an eigenvalue with strictly positive real part, and hence the equilibrium $\overline{x}_e$ of the vector field $X$ is unstable.
\end{proof}

Next we show that if $\mathcal{I}_{X}(\overline{x}_e)>0$, then the non-degenerate regular equilibrium point $\overline{x}_e$ is Lyapunov stable. In order to do that, we will use the \textit{Arnold stability test}.

Before stating the stability result, let us recall the \textit{Arnold stability test}. For more details regarding the proof of the Arnold stability test see, e.g., \cite{arnstb}.
\begin{theorem}[\cite{arnstb}](Arnold stability test)\label{AST}
Let $M$ be a smooth manifold, $X\in\mathfrak{X}(M)$ a smooth vector field, and $m\in M$ an equilibrium point of $X$. Assume that $X$ admits $k+1$ first integrals, $H, C_1,\dots,C_k \in \mathcal{C}^{2}(M,\mathbb{R})$, and moreover, there exist some real numbers $\lambda^{e}_1,\dots,\lambda^{e}_k \in\mathbb{R}$ such that
\begin{align}\label{condAr}
\begin{split}
&\mathrm{d}(H+\lambda^{e}_1 C_1 +\dots + \lambda^{e}_k C_k)(m)=0,~ \text{and}\\
&\mathrm{d}^2(H+\lambda^{e}_1 C_1 +\dots + \lambda^{e}_k C_k)(m)|_{W\times W} \ \textit{is positive or negative definite},
\end{split}
\end{align}
where $W:=\ker \mathrm{d}C_1 (m)\cap \dots \cap \ker \mathrm{d}C_k (m)$. Then the equilibrium point $m$ is Lyapunov stable.
\end{theorem}

Let us state now the main stability result of this article, which together with Theorem \eqref{spec1} generates a new criterion to test the stability of non-degenerate equilibrium states of completely integrable systems.

\begin{theorem}\label{mainthm2}
Let $\overline{x}_e \in \mathcal{E}^{C_{n-1}}_{C_1 ,\dots, C_{n-2}}$ be a non-degenerate regular equilibrium point of the vector field $X$ written in the form \eqref{hax}, and let $\mathcal{I}_{X}(\overline{x}_e)$ be the associated scalar quantity introduced in Definition \eqref{invi}. If $\mathcal{I}_{X}(\overline{x}_e)>0$, then the equilibrium state $\overline{x}_e$ is Lyapunov stable.
\end{theorem}
\begin{proof} Let us start by recalling from Theorem \eqref{darbnf} the existence of a smooth diffeomorphism $\Phi$ generated by a triple $(\Omega^{\prime},\Phi_1,\Phi_2)$ consisting of an open neighborhood $\Omega^{\prime}\subseteq \Omega$ of $\overline{x}_e$, and two smooth functions $\Phi_1,\Phi_2 \in\mathcal{C}^{\infty}(\Omega^{\prime},\mathbb{R})$, such that the map $\Phi:\Omega^{\prime}\rightarrow W^{\prime}=\Phi(\Omega^{\prime})$, given by $\Phi=(\Phi_1,\Phi_2,C_1,\dots,C_{n-2})$, is a smooth diffeomorphism, and $\nu(\overline{x})\neq 0$, for every $\overline{x}\in\Omega^{\prime}$.

Now, if $\{F_{t}^X\}_{t}$ stands for the flow of $X$, then the flow of $\Phi_{\star}X$ is given by $\{\Phi\circ F_{t}^X \circ \Phi^{-1}\}_{t}$, and so, the map $t\mapsto F_{t}^X (\overline{x})$, is the integral curve of $X$ starting from   $\overline{x}\in\Omega^{\prime}$, if and only if the map $t\mapsto \Phi (F_{t}^X (\overline{x}))$, is the integral curve of $\Phi_{\star}X$ starting from $\Phi(\overline{x})\in W^{\prime}:=\Phi(\Omega^{\prime})$. 

Consequently, the equilibrium state $\overline{x}_e$ of the vector field $X$ is Lyapunov stable if and only if the equilibrium state $\Phi(\overline{x}_e)$ of the vector field $\Phi_{\star}X$ is Lyapunov stable. 

Let us show now that the equilibrium state $\Phi(\overline{x}_e)$ of the vector field $\Phi_{\star}X$, verifies the conditions of  Arnold stability test, and hence is Lyapunov stable. 

In order to do that, recall from Theorem \eqref{tsv} that $\Phi_{\star}H$, $\Phi_{\star} C_1$, $\dots$, $\Phi_{\star}C_{n-2}\in \mathcal{C}^{\infty}(W^{\prime},\mathbb{R})$ are first integrals of the vector field $\Phi_{\star}X$. 

As $\overline{x}_e \in \mathcal{E}^{C_{n-1}}_{C_1 ,\dots, C_{n-2}}$ is a non-degenerate regular equilibrium point of system \eqref{hax}, there exists $\overrightarrow{\lambda_{e}}:=(\lambda^{e}_1, \dots,\lambda^{e}_{n-2})\in\mathbb{R}^{n-2}$ such that $\overline{x}_e$ is a non-degenerate critical point of the smooth function $F_{\overrightarrow{\lambda_{e}}}:=H+\lambda^{e}_1 C_1 +\dots + \lambda^{e}_{n-2} C_{n-2}$.

Next we prove that the smooth function $\Phi_{\star}F_{\overrightarrow{\lambda_{e}}}:=\Phi_{\star}H + \lambda^{e}_1 \Phi_{\star}C_1 +\dots +\lambda^{e}_{n-2} \Phi_{\star}C_{n-2},$ verifies the conditions of Arnold stability test.

First condition of Arnold's stability test \eqref{condAr} follows directly from Proposition \eqref{lemma1}, since $\Phi(\overline{x}_e)$ is a non-degenerate critical point of $\Phi_{\star}F_{\overrightarrow{\lambda_{e}}}$.

In order to verify the second condition of Arnold's stability test, note that the gradient defining relation, i.e.,
$$
\langle\nabla(\Phi_{\star}C_i)(\Phi(\overline{x}_e)),u\rangle=\mathrm{d}(\Phi_{\star}C_i)(\Phi(\overline{x}_e))\cdot u, \ \text{for all} \ u\in T_{\Phi(\overline{x}_e)}W^{\prime}=T_{\Phi(\overline{x}_e)}\mathbb{R}^{n}\cong \mathbb{R}^{n},
$$
valid for each $i\in\{1,\dots, n-2\}$, together with the equality \eqref{casgrad}, i.e., for each $i\in\{1,\dots, n-2\}$, 
\begin{equation*}
\nabla(\Phi_{\star}C_i)(\Phi(\overline{x}_e))=\dfrac{\partial}{\partial y_{i+2}}|_{\Phi(\overline{x}_e)}\in T_{\Phi(\overline{x}_e)}W^{\prime}=T_{\Phi(\overline{x}_e)}\mathbb{R}^{n}\cong e_{i+2}\in\mathbb{R}^{n},
\end{equation*}
imply that 
\begin{align*}
W:&=\ker \mathrm{d}(\Phi_{\star}C_1 )(\Phi(\overline{x}_e))\cap \dots \cap \ker \mathrm{d}(\Phi_{\star}C_{n-2} )(\Phi(\overline{x}_e))\\
&=\operatorname{span}_{\mathbb{R}}\left\{ \dfrac{\partial}{\partial y_{1}}|_{\Phi(\overline{x}_e)},  \dfrac{\partial}{\partial y_{2}}|_{\Phi(\overline{x}_e)}\right\}\cong \operatorname{span}_{\mathbb{R}}\left\{ e_1 , e_2 \right\}.
\end{align*}
Taking into account the formula \eqref{relimp2}, one obtains the following matrix representation of the bilinear form $\mathrm{d}^2 (\Phi_{\star}F_{\overrightarrow{\lambda_{e}}})(\Phi(\overline{x}_e))|_{W\times W}$ with respect to the basis $\left\{ e_1 , e_2 \right\}$ of $W$:
\begin{align}\label{arn2}
\begin{split}
\mathrm{d}^2 (\Phi_{\star}F_{\overrightarrow{\lambda_{e}}})(\Phi(\overline{x}_e))|_{W\times W}&=
\begin{bmatrix}
I_2 \ | \ O_{2,n-2}\\
\end{bmatrix}\cdot   
\left[
\begin{array}{c|c}
H_{2,2}(\Phi(\overline{x}_e)) & H_{2,n-2}(\Phi(\overline{x}_e))\\ 
\hline
H_{n-2,2}(\Phi(\overline{x}_e)) &  H_{n-2,n-2}(\Phi(\overline{x}_e))
\end{array}\right]\cdot
\begin{bmatrix}
I_2 \\ 
\hline
O_{n-2,2}
\end{bmatrix}\\
&=H_{2,2}(\Phi(\overline{x}_e)),
\end{split}
\end{align}
where the blocks $H_{k,l}(\Phi(\overline{x}_e))\in\mathcal{M}_{k,l}(\mathbb{R})$, $k,l\in\{2,n-2\}$, are given by
\begin{align*}
H_{2,2}(\Phi(\overline{x}_e)):&=\left[ \dfrac{\partial^2 (\Phi_{\star}H)}{\partial y_i \partial y_j}(\Phi(\overline{x}_e))\right]_{1\leq i,j\leq 2},  H_{2,n-2}(\Phi(\overline{x}_e)):=\left[ \dfrac{\partial^2 (\Phi_{\star}H)}{\partial y_i \partial y_j}(\Phi(\overline{x}_e)) \right]_{1\leq i\leq 2, 3\leq j\leq n}\\
H_{n-2,2}(\Phi(\overline{x}_e)):&=\left[ \dfrac{\partial^2 (\Phi_{\star}H)}{\partial y_i \partial y_j}(\Phi(\overline{x}_e))\right]_{3\leq i\leq n, 1\leq j\leq 2},  H_{n-2,n-2}(\Phi(\overline{x}_e)):=\left[ \dfrac{\partial^2 (\Phi_{\star}H)}{\partial y_i \partial y_j}(\Phi(\overline{x}_e)) \right]_{3\leq i,j\leq n}.
\end{align*}

Let us show now that $\det(H_{2,2}(\Phi(\overline{x}_e)))>0$. Indeed, from the relation \eqref{ifeq}, the Theorem \eqref{frstmain}, and the hypothesis $\mathcal{I}_{X}(\overline{x}_e)>0$, we obtain that
\begin{equation}\label{cimpo}
\nu^2_{\Phi}(\Phi(\overline{x}_e))\cdot \det(H_{2,2}(\Phi(\overline{x}_e)))=\mathcal{I}_{\Phi_{\star}X}(\Phi(\overline{x}_e))=\mathcal{I}_{X}(\overline{x}_e)>0.
\end{equation}
Since $\nu_{\Phi}(\Phi(\overline{x}_e))\neq 0$ (from Proposition \eqref{lemma1}), the above relation implies that  $$\det(H_{2,2}(\Phi(\overline{x}_e)))>0.$$

As $\mathrm{d}^2 (\Phi_{\star}F_{\overrightarrow{\lambda_{e}}})(\Phi(\overline{x}_e))|_{W\times W}=H_{2,2}(\Phi(\overline{x}_e))\in\mathcal{M}_{2}(\mathbb{R})$ is a $2 \times 2$ real symmetric matrix, this is positive or negative definite if and only if its determinant is strictly positive. From the relation \eqref{cimpo}, this condition is equivalent to $\mathcal{I}_{X}(\overline{x}_e)>0$, and hence the second condition of Arnold's stability test \eqref{condAr} is verified too. Consequently, as both conditions of Arnold's stability test are verified, we obtain the conclusion.

In order to complete the proof, we need to show that $H_{2,2}(\Phi(\overline{x}_e))\in\mathcal{M}_{2}(\mathbb{R})$ is positive or negative definite if and only if its determinant is strictly positive. Indeed, $H_{2,2}(\Phi(\overline{x}_e))$ is positive definite if and only if $\dfrac{\partial^2 (\Phi_{\star}H)}{\partial y^2_1}(\Phi(\overline{x}_e))>0$ and $\det(H_{2,2}(\Phi(\overline{x}_e)))>0$. On the other hand, $H_{2,2}(\Phi(\overline{x}_e))$ is negative definite if and only if $\dfrac{\partial^2 (\Phi_{\star}H)}{\partial y^2_1}(\Phi(\overline{x}_e))<0$ and $\det(H_{2,2}(\Phi(\overline{x}_e)))>0$. Hence, if $H_{2,2}(\Phi(\overline{x}_e))$ is positive or negative definite, then $\det(H_{2,2}(\Phi(\overline{x}_e)))>0$. Conversely, let us show first that if $\det(H_{2,2}(\Phi(\overline{x}_e)))>0$, then $\dfrac{\partial^2 (\Phi_{\star}H)}{\partial y^2_1}(\Phi(\overline{x}_e))\neq 0$. Indeed, assuming $\dfrac{\partial^2 (\Phi_{\star}H)}{\partial y^2_1}(\Phi(\overline{x}_e)) = 0$, it follows that
\begin{align*}
\det(H_{2,2}(\Phi(\overline{x}_e)))&=\det\left( \begin{bmatrix}
0 & \dfrac{\partial^2 (\Phi_{\star}H)}{\partial y_1 \partial y_2}(\Phi(\overline{x}_e))\\
\dfrac{\partial^2 (\Phi_{\star}H)}{\partial y_2 \partial y_1}(\Phi(\overline{x}_e)) & \dfrac{\partial^2 (\Phi_{\star}H)}{\partial y^2_2}(\Phi(\overline{x}_e))
\end{bmatrix}\right)\\
&= -\left( \dfrac{\partial^2 (\Phi_{\star}H)}{\partial y_1 \partial y_2}(\Phi(\overline{x}_e))\right)^2\leq 0,
\end{align*}
which contradicts the relation $\det(H_{2,2}(\Phi(\overline{x}_e)))>0$. Consequently, we distinguish between two possibilities, namely, either $\dfrac{\partial^2 (\Phi_{\star}H)}{\partial y^2_1}(\Phi(\overline{x}_e))> 0$ (which together with $\det(H_{2,2}(\Phi(\overline{x}_e)))>0$ implies that $H_{2,2}(\Phi(\overline{x}_e))$ is positive definite), or $\dfrac{\partial^2 (\Phi_{\star}H)}{\partial y^2_1}(\Phi(\overline{x}_e))< 0$ (which together with $\det(H_{2,2}(\Phi(\overline{x}_e)))>0$ implies that $H_{2,2}(\Phi(\overline{x}_e))$ is negative definite). Hence, if $\det(H_{2,2}(\Phi(\overline{x}_e)))>0$, then $H_{2,2}(\Phi(\overline{x}_e))$ is positive or negative definite.
\end{proof}

Let us state now the main result of this article, which provides a stability criterion for non-degenerate regular equilibrium states of the Hamiltonian realizations of completely integrable systems.
\begin{theorem}\label{mainthm}
Let $\overline{x}_e \in \mathcal{E}^{C_{n-1}}_{C_1 ,\dots, C_{n-2}}$ be a non-degenerate regular equilibrium point of the vector field $X$ realized as the Hamiltonian dynamical system \eqref{systy}, and let $\mathcal{I}_{X}(\overline{x}_e)$ be the associated scalar quantity introduced in Definition \eqref{invi}. Then the following implications hold true:
\begin{enumerate}
\item if $\mathcal{I}_{X}(\overline{x}_e)<0$, then the equilibrium state $\overline{x}_e$ is unstable,
\item if $\mathcal{I}_{X}(\overline{x}_e)>0$, then the equilibrium state $\overline{x}_e$ is Lyapunov stable.
\end{enumerate}
\end{theorem}
\begin{proof}
The proof follows directly from Theorem \eqref{spec1} and Theorem \eqref{mainthm2}.
\end{proof}
\begin{remark}
It remains an open problem to provide minimal additional conditions which imply stability/instability of the equilibrium point $\overline{x}_e$, in the degenerate case $\mathcal{I}_{X}(\overline{x}_e)=0$.
\end{remark}

Let us present now a direct consequence of Theorem \eqref{mainthm} and Theorem \eqref{ift}, regarding the stability properties 
of non-degenerate regular equilibrium states located nearby a fixed non-degenerate regular equilibrium point $\overline{x}_e$ such that $\mathcal{I}_{X}(\overline{x}_e)\neq 0$.

\begin{theorem}
Let $\overline{x}_e \in \mathcal{E}^{C_{n-1}}_{C_1 ,\dots, C_{n-2}}$ be a non-degenerate regular equilibrium point of the vector field $X$ realized as the Hamiltonian dynamical system \eqref{systy}. Let $\overrightarrow{\lambda_{e}}:=(\lambda^{e}_1,\dots,\lambda^{e}_{n-2})\in\mathbb{R}^{n-2}$ be such that $\overline{x}_e$ is a non-degenerate critical point of the smooth function $F_{\overrightarrow{\lambda_{e}}}:=C_{n-1}+\lambda^{e}_1 C_1 +\dots + \lambda^{e}_{n-2} C_{n-2}$.
Then the following implications hold true:
\begin{enumerate}
\item if $\mathcal{I}_{X}(\overline{x}_e)<0$, then there exist $V\subseteq \mathbb{R}^{n-2}$, an open neighborhood of $\overrightarrow{\lambda_{e}}$, $U\subseteq \Omega$, an open neighborhood of $\overline{x}_e$, a smooth function $\overline{x}: V\rightarrow U$ such that $\overline{x}(\overrightarrow{\lambda_{e}})=\overline{x}_e$, and for every $\overrightarrow{\lambda}\in V$, $\overline{x}(\overrightarrow{\lambda}) \in \mathcal{E}^{C_{n-1}}_{C_1 ,\dots, C_{n-2}}\cap U$ is an unstable non-degenerate regular equilibrium state of the Hamiltonian system \eqref{systy}. 
\item if $\mathcal{I}_{X}(\overline{x}_e)>0$, then there exist $V\subseteq \mathbb{R}^{n-2}$, an open neighborhood of $\overrightarrow{\lambda_{e}}$, $U\subseteq \Omega$, an open neighborhood of $\overline{x}_e$, a smooth function $\overline{x}: V\rightarrow U$ such that $\overline{x}(\overrightarrow{\lambda_{e}})=\overline{x}_e$, and for every $\overrightarrow{\lambda}\in V$, $\overline{x}(\overrightarrow{\lambda}) \in \mathcal{E}^{C_{n-1}}_{C_1 ,\dots, C_{n-2}} \cap U$ is a Lyapunov stable non-degenerate regular equilibrium state of the Hamiltonian system \eqref{systy}. 
\end{enumerate}
\end{theorem}
\begin{proof}
From Theorem \eqref{ift}, there exist $V_{1} \subseteq \mathbb{R}^{n-2}$, an open neighborhood of $\overrightarrow{\lambda_{e}}$, $U\subseteq \Omega$, an open neighborhood of $\overline{x}_e$, and a smooth function $\overline{x}: V_{1} \rightarrow U$ such that $\overline{x}(\overrightarrow{\lambda_{e}})=\overline{x}_e$, and moreover, for each $\overrightarrow{\lambda}\in V_{1}$, $\overline{x}(\overrightarrow{\lambda}) \in \mathcal{E}^{C_{n-1}}_{C_1 ,\dots, C_{n-2}} \cap U$ is a non-degenerate regular equilibrium state of the integrable system \eqref{systy}. Assume that $\mathcal{I}_{X}(\overline{x}_e)\neq 0$. Since the function $\mathcal{I}_{X}\circ\overline{x}:V_1 \rightarrow \mathbb{R}$ is continuous, there exists $V\subseteq V_1$ an open neighborhood of $\overrightarrow{\lambda_{e}}$ such that $\operatorname{sgn}({\mathcal{I}_{X}(\overline{x}_e)})=\operatorname{sgn}({\mathcal{I}_{X}(\overline{x}(\overrightarrow{\lambda_{e}}))}) =\operatorname{sgn}({\mathcal{I}_{X}(\overline{x}(\overrightarrow{\lambda}))})$, for every $\overrightarrow{\lambda}\in V$. Now the conclusion follows from Theorem \eqref{mainthm} applied to each non-degenerate regular equilibrium point $\overline{x}(\overrightarrow{\lambda})$,  $\overrightarrow{\lambda}\in V$. 
\end{proof}

\section{Leafwise stability of non-degenerate equilibria of completely integrable systems}

The aim of this section is to provide a criterion to decide leafwise stability of non-degenerate regular equilibria of Hamiltonian realizations of completely integrable systems. In order to do that we fix a non-degenerate regular equilibrium point of the completely integrable system \eqref{sys} realized as the Hamiltonian system \eqref{systy}, and we choose an open neighborhood around the equilibrium point, where the Darboux Normal Form Theorem \eqref{darbnf} can be applied. More precisely, let $$X = \sum_{i=1}^{n}\nu \cdot \dfrac{\partial(C_1,\dots,C_{n-2},x_i,H)}{\partial(x_1,\dots,x_n)}\cdot\dfrac{\partial}{\partial{x_i}},$$ be the vector field associated to the completely integrable system \eqref{sys}, realized as the Hamiltonian dynamical system \eqref{systy}, i.e., $\left(\Omega,\{\cdot,\cdot\}_{\nu;C_1,\dots,C_{n-2}},H=C_{n-1}\right)$. Let $\overline{x}_e \in \mathcal{E}^{C_{n-1}}_{C_1 ,\dots, C_{n-2}}$ be a non-degenerate regular equilibrium point of $X$. Let $(\Omega^{\prime},\Phi_1,\Phi_2)$ be a triple as introduced in Theorem \eqref{darbnf}, consisting of an open neighborhood $\Omega^{\prime}\subseteq \Omega$ of $\overline{x}_e$, and two smooth functions $\Phi_1,\Phi_2 \in\mathcal{C}^{\infty}(\Omega^{\prime},\mathbb{R})$, such that the map $\Phi:\Omega^{\prime}\rightarrow W^{\prime}=\Phi(\Omega^{\prime})$, given by $\Phi=(\Phi_1,\Phi_2,C_1,\dots,C_{n-2})$, is a smooth diffeomorphism, and $\nu(\overline{x})\neq 0$, for every $\overline{x}\in\Omega^{\prime}$. Then, $\Phi_{\star}X$, the push forward of the vector field $X$ by $\Phi$, is a Hamiltonian vector field, with Hamiltonian $\Phi_{\star}H=\Phi_{\star}C_{n-1}$, defined on the Poisson manifold $\left(W^{\prime},\{\cdot,\cdot\}_{\nu_{\Phi};\Phi_{\star}C_1,\dots,\Phi_{\star}C_{n-2}}\right)$, and has the expression
\begin{equation}\label{foleq}
\Phi_{\star}X = \nu_{\Phi}\cdot\left[ \dfrac{\partial(\Phi_{\star}H)}{\partial y_2}\cdot\dfrac{\partial}{\partial y_1}-\dfrac{\partial(\Phi_{\star}H)}{\partial y_1}\cdot\dfrac{\partial}{\partial y_2}\right],
\end{equation}
where $\nu_{\Phi}=\Phi_{\star}\nu\cdot \Phi_{\star}\operatorname{Jac}(\Phi)$, and $(y_1,\dots,y_n)=\Phi(x_1,\dots,x_n)$, denote the local coordinates on $W^{\prime}$. Moreover, as the diffeomorphism $\Phi:\Omega^{\prime}\rightarrow W^{\prime}=\Phi(\Omega^{\prime})$ is also a Poisson isomorphism between the Poisson manifolds $\left(\Omega^{\prime},\{\cdot,\cdot\}_{\nu;C_1,\dots,C_{n-2}}\right)$ and $\left(W^{\prime},\{\cdot,\cdot\}_{\nu_{\Phi};\Phi_{\star}C_1,\dots,\Phi_{\star}C_{n-2}}\right)$, it maps symplectic leaves to symplectic leaves. Let us denote by $\Sigma^{\prime}_{\overline{x}_e}\subset \Omega^{\prime}$ the (regular) symplectic leaf of the Poisson manifold $\left(\Omega^{\prime},\{\cdot,\cdot\}_{\nu;C_1,\dots,C_{n-2}}\right)$, which contains the non-degenerate regular equilibrium point $\overline{x}_e$. Since $Z(\nu|_{\Omega^{\prime}})=\emptyset$, it follows from the second section of the article that $\Sigma^{\prime}_{\overline{x}_e}$ is the connected component which contains $\overline{x}_e$, of the two-dimensional manifold given by $C^{-1}(\{(c_1,\dots,c_{n-2})\})$, if $C(\overline{x}_e)=:(c_1,\dots,c_{n-2})$ is a regular value of $C:=({C_1}|_{\Omega^{\prime}} ,\dots, {C_{n-2}}|_{\Omega^{\prime}})$, or given by $C^{-1}(\{(c_1,\dots,c_{n-2})\})\setminus \operatorname{Crit}(C)$, if $(c_1,\dots,c_{n-2})$ is a critical value of $C$, where $\operatorname{Crit}(C)\subset \Omega^{\prime}$ stands for the set of critical points of $C$.

Consequently, since $\Phi:\Omega^{\prime}\rightarrow W^{\prime}=\Phi(\Omega^{\prime})$ is a Poisson isomorphism, the non-degenerate regular equilibrium $\overline{x}_e$ is Lyapunov stable (unstable) relative to perturbations along the symplectic leaf $\Sigma^{\prime}_{\overline{x}_e}$ of the Poisson manifold $\left(\Omega^{\prime},\{\cdot,\cdot\}_{\nu;C_1,\dots,C_{n-2}}\right)$, if and only if the non-degenerate regular equilibrium $\Phi(\overline{x}_e)$ is Lyapunov stable (unstable) relative to perturbations along the symplectic leaf $\Phi(\Sigma^{\prime}_{\overline{x}_e})$ of the Poisson manifold $\left(W^{\prime},\{\cdot,\cdot\}_{\nu_{\Phi};\Phi_{\star}C_1,\dots,\Phi_{\star}C_{n-2}}\right)$. Otherwise stated, $\overline{x}_e$ is a Lyapunov stable (unstable) non-degenerate regular equilibrium of the vector field $X|_{\Sigma^{\prime}_{\overline{x}_e}}$ (where the vector field $X$ is identified with its restriction $X|_{\Omega^{\prime}}\in\mathfrak{X}(\Omega^{\prime})$) if and only if $\Phi(\overline{x}_e)$ is a Lyapunov stable (unstable) non-degenerate regular equilibrium of the vector field $\Phi_{\star}( X|_{\Sigma^{\prime}_{\overline{x}_e}})$. Since $\Phi_{\star}( X|_{\Sigma^{\prime}_{\overline{x}_e}})=(\Phi_{\star}X)|_{\Phi(\Sigma^{\prime}_{\overline{x}_e})}$, one can reduce the original stability problem, to the stability analysis of the non-degenerate regular equilibrium $\Phi(\overline{x}_e)$ of the vector field $(\Phi_{\star}X)|_{\Phi(\Sigma^{\prime}_{\overline{x}_e})}$. Note that since $\Phi$ is a Poisson isomorphism, then $\Phi(\Sigma^{\prime}_{\overline{x}_e})$ is the regular symplectic leaf of $\left(W^{\prime},\{\cdot,\cdot\}_{\nu_{\Phi};\Phi_{\star}C_1,\dots,\Phi_{\star}C_{n-2}}\right)$ which contains the equilibrium $\Phi(\overline{x}_e)$ of $\Phi_{\star}X$. More precisely, $\Phi(\Sigma^{\prime}_{\overline{x}_e})$ is the connected component of the two-dimensional manifold $W^{\prime}\bigcap\{(y_1, y_2, \dots,y_n )\in\mathbb{R}^{n} : y_3 =c_1 ,\dots, y_{n}=c_{n-2}\}$ which contains $\Phi(\overline{x}_e)$, where $(c_1,\dots,c_{n-2})=(C_1 (\overline{x}_e),\dots, C_{n-2} (\overline{x}_e))$. Using the relation \eqref{foleq}, we obtain the following local expression of the system of ordinary differential equations induced by the two-dimensional symplectic Hamiltonian vector field $(\Phi_{\star}X)|_{\Phi(\Sigma^{\prime}_{\overline{x}_e})}$: 
\begin{equation}\label{foleqok}
\left\{\begin{array}{l}
\dot y_{1}=\nu_{\Phi;c_1,\dots,c_{n-2}}\cdot \dfrac{\partial H_{c_1,\dots,c_{n-2}}}{\partial y_2}\\
\dot y_{2}= \nu_{\Phi;c_1,\dots,c_{n-2}}\cdot \left( -\dfrac{\partial H_{c_1,\dots,c_{n-2}}}{\partial y_1}\right),\\
\end{array}\right.
\end{equation}
where $$\nu_{\Phi;c_1,\dots,c_{n-2}}(y_1,y_2):=\nu_{\Phi}(y_1,y_2,c_1,\dots,c_{n-2}), H_{c_1,\dots,c_{n-2}}(y_1,y_2):=(\Phi_{\star}H)(y_1,y_2,c_1,\dots,c_{n-2})$$ for every $(y_1,y_2,c_1,\dots,c_{n-2})\in \Phi(\Sigma^{\prime}_{\overline{x}_e})$. Since $\nu (x)\neq 0$ for every $x\in\Omega^{\prime}$, we obtain that $\nu_{\Phi}=\Phi_{\star}\nu\cdot \Phi_{\star}\operatorname{Jac}(\Phi)$ is nonvanishing in $W^{\prime}$. Consequently, using the relation \eqref{foleqok}, it follows that the $2-$form $\nu_{\Phi}(y_1,y_2,c_1,\dots,c_{n-2})\cdot\mathrm{d}y_1 \wedge \mathrm{d}y_2$, is the symplectic form on $\Phi(\Sigma^{\prime}_{\overline{x}_e})$ induced by the Poisson bracket $\{\cdot,\cdot\}_{\nu_{\Phi};\Phi_{\star}C_1,\dots,\Phi_{\star}C_{n-2}}$ defined on $W^{\prime}$.

Let us recall now the relation \eqref{relimp1}, i.e., 
\begin{align*}
\dfrac{\partial(\Phi_{\star}H)}{\partial y_1}(\Phi(\overline{x}_e))=\dfrac{\partial(\Phi_{\star}H)}{\partial y_2}(\Phi(\overline{x}_e))=0.
\end{align*}
Consequently, as $\Phi(\overline{x}_e)=(\Phi_1 (\overline{x}_e), \Phi_2 (\overline{x}_e),c_1,\dots,c_{n-2})$, it follows that
\begin{align}\label{eqconditon}
\begin{split}
\dfrac{\partial H_{c_1,\dots,c_{n-2}}}{\partial y_1}(\Phi_1 (\overline{x}_e), \Phi_2 (\overline{x}_e))&=\dfrac{\partial(\Phi_{\star}H)}{\partial y_1}(\Phi(\overline{x}_e))=0,\\
\dfrac{\partial H_{c_1,\dots,c_{n-2}}}{\partial y_2}(\Phi_1 (\overline{x}_e), \Phi_2 (\overline{x}_e))&=\dfrac{\partial(\Phi_{\star}H)}{\partial y_2}(\Phi(\overline{x}_e))=0,
\end{split}
\end{align}
and so the point $(\Phi_1 (\overline{x}_e), \Phi_2 (\overline{x}_e))$ is an equilibrium state of the two-dimensional symplectic Hamiltonian system \eqref{foleqok}. Hence, in order to study the dynamics of the vector field $(\Phi_{\star}X)|_{\Phi(\Sigma^{\prime}_{\overline{x}_e})}$ around the equilibrium point $\Phi(\overline{x}_e)$, we shall study instead the dynamics of the two-dimensional symplectic Hamiltonian system \eqref{foleqok} around the equilibrium point $(\Phi_1 (\overline{x}_e), \Phi_2 (\overline{x}_e))$. 

Next result presents the local dynamics around a non-degenerate regular equilibrium point, $\overline{x}_e$, of the Hamiltonian system $\left(\Omega,\{\cdot,\cdot\}_{\nu;C_1,\dots,C_{n-2}},H=C_{n-1}\right)$, restricted to the corresponding symplectic leaf $\Sigma_{\overline{x}_e}$ of the Poisson manifold $\left(\Omega,\{\cdot,\cdot\}_{\nu;C_1,\dots,C_{n-2}}\right)$.

\begin{theorem}\label{specleaf}
Let $\overline{x}_e \in \mathcal{E}^{C_{n-1}}_{C_1 ,\dots, C_{n-2}}$ be a non-degenerate regular equilibrium point of the vector field $X$ realized as the Hamiltonian dynamical system \eqref{systy}. Let $\Sigma_{\overline{x}_e}\subset \Omega$ be the sympectic leaf of the Poisson manifold $\left(\Omega,\{\cdot,\cdot\}_{\nu;C_1,\dots,C_{n-2}}\right)$, passing through $\overline{x}_e$. Then the following assertions hold true.
\begin{itemize}
\item [(a)] The characteristic polynomial of $\mathfrak{L}^{X|_{\Sigma_{\overline{x}_e}}}(\overline{x}_e)$ is given by 
$$
p_{\mathfrak{L}^{X|_{\Sigma_{\overline{x}_e}}}(\overline{x}_e)}(\mu)= \mu^2 + \mathcal{I}_{X}(\overline{x}_e).
$$
\item [(b)] If $\mathcal{I}_{X}(\overline{x}_e)<0$, then the equilibrium state $\overline{x}_e$ is an unstable equilibrium point of the restricted vector field $X|_{\Sigma_{\overline{x}_e}}$.

\item [(c)] If $\mathcal{I}_{X}(\overline{x}_e)>0$, then the equilibrium state $\overline{x}_e$ is a Lyapunov stable equilibrium point of the restricted vector field $X|_{\Sigma_{\overline{x}_e}}$. 

\item [(d)] If $\mathcal{I}_{X}(\overline{x}_e)>0$, then there exists $\varepsilon_{0}>0$ and a one-parameter family of periodic orbits of $X|_{\Sigma_{\overline{x}_e}}$ (and hence of $X$ too), $\left\{\gamma_{\varepsilon}\right\}_{0<\varepsilon\leq\varepsilon_0}\subset \Sigma_{\overline{x}_e}$, that shrink to $\overline{x}_e$ as $\varepsilon\rightarrow 0$, with periods $T_{\varepsilon}\rightarrow{\frac{2\pi}{\sqrt{\mathcal{I}_{X}(\overline{x}_e)}}}$ as $\varepsilon\rightarrow 0$. Moreover, the set $\{\overline{x}_e\}\cup\bigcup_{0<\varepsilon < \varepsilon_0} \gamma_{\varepsilon}$ represents the connected component of  $\Sigma_{\overline{x}_e}\setminus \gamma_{\varepsilon_{0}}$, which contains the equilibrium point $\overline{x}_e$.
\end{itemize}
\end{theorem}
\begin{proof}
In order to prove the Theorem we shall restrict our analysis to an open neighborhood $\Omega^{\prime}\subseteq \Omega$ around $\overline{x}_e$, where the Darboux Normal Form holds true. Consequently, using the above notations, we shall consider in the following, the local dynamics of the vector field $X|_{\Sigma^{\prime}_{\overline{x}_e}}$ around the equilibrium point $\overline{x}_e$.
\begin{itemize}  
\item [(a)] As $\mathfrak{L}^{X|_{\Sigma_{\overline{x}_e}}}(\overline{x}_e)=\mathfrak{L}^{X|_{\Sigma^{\prime}_{\overline{x}_e}}}(\overline{x}_e)$, and $\mathfrak{L}^{X|_{\Sigma^{\prime}_{\overline{x}_e}}}(\overline{x}_e)=\mathfrak{L}^{(\Phi_{\star}X)|_{\Phi(\Sigma^{\prime}_{\overline{x}_e})}}(\Phi(\overline{x}_e))$ (by Proposition \eqref{diflin}), it follows that $\mathfrak{L}^{X|_{\Sigma_{\overline{x}_e}}}(\overline{x}_e)$ represents the linearization of the system \eqref{foleqok} evaluated at the equilibrium point $(\Phi_1 (\overline{x}_e), \Phi_2 (\overline{x}_e))$. Consequently, we have that
\begin{align*}
\mathfrak{L}^{X|_{\Sigma_{\overline{x}_e}}}(\overline{x}_e) = &\nu_{\Phi;c_1,\dots,c_{n-2}}(\Phi_1 (\overline{x}_e), \Phi_2 (\overline{x}_e))\\
& \cdot
\begin{bmatrix}
\dfrac{\partial^2 H_{c_1,\dots,c_{n-2}}}{\partial y_2 \partial y_1}(\Phi_1 (\overline{x}_e), \Phi_2 (\overline{x}_e)) & \dfrac{\partial^2 H_{c_1,\dots,c_{n-2}}}{\partial y_2 ^2}(\Phi_1 (\overline{x}_e), \Phi_2 (\overline{x}_e))\\
- \dfrac{\partial^2 H_{c_1,\dots,c_{n-2}}}{\partial y_1^2}(\Phi_1 (\overline{x}_e), \Phi_2 (\overline{x}_e)) & - \dfrac{\partial^2 H_{c_1,\dots,c_{n-2}}}{\partial y_1 \partial y_2}(\Phi_1 (\overline{x}_e), \Phi_2 (\overline{x}_e))\\
\end{bmatrix}\\
= &\nu_{\Phi}(\Phi(\overline{x}_e))\cdot
\begin{bmatrix}
\dfrac{\partial^2 (\Phi_{\star}H)}{\partial y_2 \partial y_1}(\Phi(\overline{x}_e)) & \dfrac{\partial^2 (\Phi_{\star}H)}{\partial y_2 ^2}(\Phi(\overline{x}_e))\\
- \dfrac{\partial^2 (\Phi_{\star}H)}{\partial y_1^2}(\Phi(\overline{x}_e)) & - \dfrac{\partial^2 (\Phi_{\star}H)}{\partial y_1 \partial y_2}(\Phi(\overline{x}_e))\\
\end{bmatrix},
\end{align*}
and hence we obtain the following expression for the associated characteristic polynomial:
\begin{align*}
\begin{split}
p_{\mathfrak{L}^{X|_{\Sigma_{\overline{x}_e}}}(\overline{x}_e))}(\mu )=&\mu^2 +\nu^2_{\Phi}(\Phi(\overline{x}_e))\cdot\left[ \dfrac{\partial^2 (\Phi_{\star}H)}{\partial y_1^2}(\Phi(\overline{x}_e)) \cdot\dfrac{\partial^2 (\Phi_{\star}H)}{\partial y_2^2}(\Phi(\overline{x}_e))\right.\\
&-\left.\left( \dfrac{\partial^2 (\Phi_{\star}H)}{\partial y_1 \partial y_2}(\Phi(\overline{x}_e))\right)^2\right].
\end{split}
\end{align*}
Using the equality \eqref{ifeq} followed by Theorem \eqref{frstmain}, we obtain
\begin{align*}
\begin{split}
p_{\mathfrak{L}^{X|_{\Sigma_{\overline{x}_e}}}(\overline{x}_e))}(\mu )&=\mu^2 +\nu^2_{\Phi}(\Phi(\overline{x}_e))\cdot\left[ \dfrac{\partial^2 (\Phi_{\star}H)}{\partial y_1^2}(\Phi(\overline{x}_e)) \cdot\dfrac{\partial^2 (\Phi_{\star}H)}{\partial y_2^2}(\Phi(\overline{x}_e))\right.\\
&-\left.\left( \dfrac{\partial^2 (\Phi_{\star}H)}{\partial y_1 \partial y_2}(\Phi(\overline{x}_e))\right)^2\right] = \mu^2 +\mathcal{I}_{\Phi_{\star}X}(\Phi(\overline{x}_e))\\
&=\mu^2 +\mathcal{I}_{X}(\overline{x}_e).
\end{split}
\end{align*}

\item [(b)] The proof follows by the fact that the characteristic polynomial of $\mathfrak{L}^{X|_{\Sigma_{\overline{x}_e}}}(\overline{x}_e)$ admits the strictly positive root $\mu_{+} = \sqrt{-\mathcal{I}_{X}(\overline{x}_e)}$.

\item [(c)] We shall prove equivalently that $(\Phi_1 (\overline{x}_e), \Phi_2 (\overline{x}_e))$ is a Lyapunov stable equilibrium point of the symplectic Hamiltonian dynamical system \eqref{foleqok}. In order to do that, we will use Dirichlet's Stability Theorem which states that \textit{given a symplectic Hamiltonian system, each isolated local minima/maxima of the Hamiltonian function, is a Lyapunov stable equilibrium point of the system}. Let us recall first from the relation \eqref{eqconditon} that $(\Phi_1 (\overline{x}_e), \Phi_2 (\overline{x}_e))$ is a critical point of the Hamiltonian function $H_{c_1,\dots,c_{n-2}}$. In order to check Dirichlet's Stability Theorem hypothesis, we shall show that $\mathrm{d}^{2}H_{c_1,\dots,c_{n-2}}(\Phi_1 (\overline{x}_e), \Phi_2 (\overline{x}_e))$ is positive or negative definite. As $\mathrm{d}^{2}H_{c_1,\dots,c_{n-2}}(\Phi_1 (\overline{x}_e), \Phi_2 (\overline{x}_e))$ is represented by a $2 \times 2$ symmetric real matrix, this will be positive or negative definite if and only if
\begin{equation}\label{conry}
\det(\mathrm{d}^{2}H_{c_1,\dots,c_{n-2}}(\Phi_1 (\overline{x}_e), \Phi_2 (\overline{x}_e)))>0.
\end{equation}
In order to prove the relation \eqref{conry}, let us compute first $\mathrm{d}^{2}H_{c_1,\dots,c_{n-2}}(\Phi_1 (\overline{x}_e), \Phi_2 (\overline{x}_e))$. A straightforward computation leads to 
\begin{align*}
\mathrm{d}^{2}H_{c_1,\dots,c_{n-2}}&(\Phi_1 (\overline{x}_e), \Phi_2 (\overline{x}_e))= \\
=&\begin{bmatrix}
\dfrac{\partial^2 H_{c_1,\dots,c_{n-2}}}{\partial y_1 ^2 }(\Phi_1 (\overline{x}_e), \Phi_2 (\overline{x}_e)) & \dfrac{\partial^2 H_{c_1,\dots,c_{n-2}}}{\partial y_1 \partial y_2}(\Phi_1 (\overline{x}_e), \Phi_2 (\overline{x}_e))\\
 \dfrac{\partial^2 H_{c_1,\dots,c_{n-2}}}{\partial y_2 \partial y_1}(\Phi_1 (\overline{x}_e), \Phi_2 (\overline{x}_e)) &  \dfrac{\partial^2 H_{c_1,\dots,c_{n-2}}}{\partial y_2 ^2}(\Phi_1 (\overline{x}_e), \Phi_2 (\overline{x}_e))\\
\end{bmatrix}\\
= &
\begin{bmatrix}
\dfrac{\partial^2 (\Phi_{\star}H)}{\partial y_1 ^2}(\Phi(\overline{x}_e)) & \dfrac{\partial^2 (\Phi_{\star}H)}{\partial y_1 \partial y_2}(\Phi(\overline{x}_e))\\
 \dfrac{\partial^2 (\Phi_{\star}H)}{\partial y_2 \partial y_1}(\Phi(\overline{x}_e)) &  \dfrac{\partial^2 (\Phi_{\star}H)}{\partial y_2 ^2}(\Phi(\overline{x}_e))\\
\end{bmatrix}
=: H_{2,2}(\Phi(\overline{x}_e)).
\end{align*}
Using the relation \eqref{cimpo}, and the fact that $\nu_{\Phi}(\Phi(\overline{x}_e))\neq 0$, we obtain that
\begin{align*}
\operatorname{sgn}(\det(\mathrm{d}^{2}H_{c_1,\dots,c_{n-2}}(\Phi_1 (\overline{x}_e), \Phi_2 (\overline{x}_e))))=\operatorname{sgn}(\det(H_{2,2}(\Phi(\overline{x}_e))))=\operatorname{sgn}(\mathcal{I}_{X}(\overline{x}_e))=1,
\end{align*}
and hence $\det(\mathrm{d}^{2}H_{c_1,\dots,c_{n-2}}(\Phi_1 (\overline{x}_e), \Phi_2 (\overline{x}_e)))>0$.

\item [(d)] The proof follows directly from item $(a)$ and the Lyapunov Center Theorem. In order to have a self-contained presentation, let us recall now the statement of the Lyapunov Center Theorem (for details see, e.g., \cite{abrahammarsden}). 

\textit{(Lyapunov Center Theorem) Let $(M,\omega)$ be a symplectic manifold, and let $X_H \in\mathfrak{X}(M)$ be a smooth (symplectic) Hamiltonian vector field. Assume that $m_e \in M$ is an equilibrium point of $X_{H}$ such that $\mu_{\pm}:=\pm i \alpha$, $\alpha >0$, are purely imaginary eigenvalues of $\mathfrak{L}^{X_H}(m_e)$, and moreover, do not exist any $k\in\mathbb{N}\setminus\{1\}$, such that $k\cdot \mu_{\pm}$ are eigenvalues of $\mathfrak{L}^{X_H}(m_e)$. Then, there exist $\varepsilon_0 >0$ and a one-parameter family of periodic orbits of $X_H$, $\left\{ \gamma_{\varepsilon}\right\}_{0<\varepsilon \leq \varepsilon_0}$ with periods $T_{\varepsilon}\rightarrow 2\pi/{\alpha}$ as $\varepsilon \rightarrow 0$. Moreover, $\gamma_{\varepsilon}$ approaches $m_e$ as $\varepsilon \rightarrow 0$, and the set $\{m_e\}\cup\bigcup_{0<\varepsilon \leq \varepsilon_0}\gamma_{\varepsilon}$ is a smooth two-dimensional submanifold with boundary $\gamma_{\varepsilon_0}$, diffeomorphic with the two-dimensional closed unit disk.}
\end{itemize}
\end{proof}

\section{Example}

In this section we illustrate the main theoretical results of this article, in the case of a concrete dynamical system coming from geophysics, which describes the irregular polarity switching of Earth's magnetic field. More precisely, the system we consider in the sequel, is the Hamiltonian version of the Rikitake two-disk dynamo system (see e.g. \cite{riki}, \cite{cookroberts}) analyzed in \cite{tudoranSIAM}, and described by the equations:
\begin{equation}\label{rik}
\left\{ \begin{array}{l}
 \dot x = yz + \beta y \\
 \dot y = xz - \beta x \\
 \dot z =  - xy \\
 \end{array} \right.
\end{equation}
where $\beta\in\mathbb{R}\setminus\{0\}$ is a parameter. The above system is the Rikitake system introduced in \cite{llibre} in the particular case $\alpha=\mu=0$. Let us denote by $X_{\beta}\in\mathfrak{X}(\mathbb{R}^3)$ the vector field which generates the Rikitake system \eqref{rik}, i.e., 
$$
X_{\beta}(x,y,z):=\left(yz + \beta y\right)\dfrac{\partial}{\partial {x}} + \left( xz - \beta x\right) \dfrac{\partial}{\partial {y}} -xy \dfrac{\partial}{\partial {z}}, ~ (\forall)(x,y,z)\in\mathbb{R}^3.
$$
As the purpose of this work concerns the stability analysis of equilibrium states, let us recall from \cite{tudoranSIAM} that the equilibria of the vector field $X_\beta$ are the elements of the set
$$
\mathcal{E}^{X_{\beta}}:=\{(M,0,\beta) : M\in\mathbb{R}\} \cup \{(0,M,-\beta) : M\in\mathbb{R}\}\cup \{(0,0,M) : M\in\mathbb{R}\}.
$$

As the system \eqref{rik} is completely integrable, admitting the first integrals $I_1, I^{\beta}_2 \in\mathcal{C}^{\infty}(\mathbb{R}^3,\mathbb{R})$ given by $I_1 (x,y,z)=\dfrac{1}{2}\left(x^2 +y^2 \right)+z^2 $, $I^{\beta}_2 (x,y,z)=\dfrac{1}{4}\left(-x^2 +y^2\right)-\beta z$, for all $(x,y,z)\in\mathbb{R}^3$, the vector field $X_{\beta}$ can be realized as a vector field of the type \eqref{sywedge} in two different ways.

\begin{enumerate}
\item The first realization of the vector field $X_\beta$ is given by
\begin{equation}\label{rikiq}
X_{\beta} = (-\nu) \star\left(\nabla C_1 \wedge \nabla C_2\right),
\end{equation}
where $\nu\equiv 1$, $C_1 =I_1$, and $C_2 =I^{\beta}_2$.

Consequently, the vector field $X_{\beta}$ admits a Hamiltonian realization of the type \eqref{systy}, $(\mathbb{R}^3,\{\cdot,\cdot\}_{1;C},H_{\beta})$, i.e., $X_{\beta}=X_{H_{\beta}}$, where $H_{\beta}:=C_2$, $C:=C_1$, and the Poisson bracket $\{\cdot,\cdot\}_{1;C}$ is given by
$$
\{f,g\}_{1;C}\cdot \mathrm{d}x\wedge\mathrm{d}y\wedge\mathrm{d}z:= \mathrm{d}C \wedge \mathrm{d}f\wedge \mathrm{d}g,
$$
for every $f,g\in\mathcal{C}^{\infty}(\mathbb{R}^3,\mathbb{R})$. 

Regarding the symplectic foliation of the Poisson manifold $(\mathbb{R}^3, \{\cdot,\cdot\}_{1;C})$, the \textit{regular symplectic leaves} are given by the ellipsoids $\Sigma_{r}:=\{(x,y,z)\in\mathbb{R}^3 : C(x,y,z)= r\}$, $r>0$, while the (unique) \textit{singular symplectic leaf} is given by the singleton $\{(0,0,0)\}$.

After some straightforward computations, it follows that the set of \textit{regular} equilibrium points of the Rikitake system \eqref{rik}, realized as the Hamiltonian dynamical system $(\mathbb{R}^3,\{\cdot,\cdot\}_{1;C},H_{\beta})$, is given by $\mathcal{E}^{X_{\beta}}\setminus\{(0,0,0)\}$, whereas the corresponding set of \textit{non-degenerate regular} equilibrium states is 
\begin{equation}\label{regriki} 
\left\{(0,0,M) : M\in\mathbb{R}\setminus\{0\} \right\}.
\end{equation}

Equivalently, following the notations introduced in Section $3$, each non-degenerate regular equilibrium point, $(0,0,M)$, $M\neq 0$, is a \textit{non-degenerate critical point} of the smooth function $F_{\beta/(2M)}\in\mathcal{C}^{\infty}(\mathbb{R}^3,\mathbb{R})$ given by
\begin{align}\label{gbeta}
\begin{split}
F_{\beta/(2M)}(x,y,z)&=H_{\beta}(x,y,z)+\dfrac{\beta}{2M}C(x,y,z)\\
&=\dfrac{\beta -M}{4M} x^2 + \dfrac{\beta + M}{4M} y^2 + \dfrac{\beta}{2M} z^2 -\beta z,
\end{split}
\end{align}
for every $(x,y,z)\in\mathbb{R}^3$.

Consequently, the associated scalar quantity, $\mathcal{I}_{X^{\beta}}(0,0,M)$, becomes 
\begin{align}\label{invriki}
\begin{split}
\mathcal{I}_{X^{\beta}}(0,0,M)&=\nu^2(0,0,M)\cdot \det{(\operatorname{Hess}F_{\beta/(2M)}(0,0,M))}\\
&\cdot <[\operatorname{Hess}F_{\beta/(2M)}(0,0,M)]^{-1}\cdot\nabla C (0,0,M), \nabla C(0,0,M)>\\
&=1^{2}\cdot\dfrac{(\beta^2 - M^2) \beta}{4 M^3}\cdot\dfrac{4 M^3}{\beta}\\
&=\beta^2 - M^2.
\end{split}
\end{align}

\item The second realization of the vector field $X_\beta$ is given by
\begin{equation}\label{rikiq2}
X_{\beta} = (-\nu) \star\left(\nabla C_1 \wedge \nabla C_2\right),
\end{equation}
where $\nu\equiv -1$, $C_1 =I^{\beta}_2$, and $C_2 =I_1$.

In this case, the vector field $X_{\beta}$ admits a Hamiltonian realization of the type \eqref{systy}, $(\mathbb{R}^3,\{\cdot,\cdot\}_{-1;C_{\beta}},H)$, i.e., $X_{\beta}=X_{H}$, where $H:=C_2$, $C_\beta :=C_1$, and the Poisson bracket $\{\cdot,\cdot\}_{-1;C_{\beta}}$ is given by
$$
\{f,g\}_{-1;C_{\beta}}\cdot \mathrm{d}x\wedge\mathrm{d}y\wedge\mathrm{d}z:= -\mathrm{d}C_{\beta} \wedge \mathrm{d}f\wedge \mathrm{d}g,
$$
for every $f,g\in\mathcal{C}^{\infty}(\mathbb{R}^3,\mathbb{R})$. 

Regarding the symplectic foliation of the Poisson manifold $(\mathbb{R}^3, \{\cdot,\cdot\}_{-1;C_{\beta}})$, all leaves are \textit{regular} and are given by the hyperbolic paraboloids $\Sigma^{\beta}_{c}:=\{(x,y,z)\in\mathbb{R}^3 : C_{\beta}(x,y,z)= c\}$, $c\in\mathbb{R}$.

After some straightforward computations, it follows that the set of \textit{regular} equilibrium points of the Rikitake system \eqref{rik}, realized as the Hamiltonian dynamical system $(\mathbb{R}^3,\{\cdot,\cdot\}_{1;C},H_{\beta})$, coincides with the full set of equilibrium points, $\mathcal{E}^{X_{\beta}}$, whereas the corresponding set of \textit{non-degenerate regular} equilibrium points is
\begin{equation}\label{regrikiki} 
\left\{(0,0,M) : M\in\mathbb{R}\right\}.
\end{equation}

Equivalently, following the notations introduced in Section $3$, each non-degenerate regular equilibrium point, $(0,0,M)$, $M\in\mathbb{R}$, is a \textit{non-degenerate critical point} of the smooth function $F_{\beta/(2M)}\in\mathcal{C}^{\infty}(\mathbb{R}^3,\mathbb{R})$ given by
\begin{align}\label{gbeta2}
\begin{split}
F_{2 M/\beta}(x,y,z)&=H(x,y,z)+\dfrac{2 M}{\beta}C_{\beta}(x,y,z)\\
&=\dfrac{\beta -M}{2\beta} x^2 + \dfrac{\beta + M}{2\beta} y^2 + z^2 -2 M z,
\end{split}
\end{align}
for every $(x,y,z)\in\mathbb{R}^3$.

Consequently, the associated scalar quantity, $\mathcal{I}_{X^{\beta}}(0,0,M)$, becomes 
\begin{align}\label{invrikiki}
\begin{split}
\mathcal{I}_{X^{\beta}}(0,0,M)&=\nu^2(0,0,M)\cdot \det{(\operatorname{Hess}F_{2 M /\beta}(0,0,M))}\\
&\cdot <[\operatorname{Hess}F_{2 M /\beta}(0,0,M)]^{-1}\cdot\nabla C_{\beta} (0,0,M), \nabla C_{\beta} (0,0,M)>\\
&=(-1)^{2}\cdot\dfrac{2(\beta^2 - M^2)}{\beta^2}\cdot\dfrac{\beta^2}{2}\\
&=\beta^2 - M^2.
\end{split}
\end{align}
\end{enumerate}

Next, we apply the stability criterion introduced in Theorem \eqref{mainthm}, in order to determine the stability of the equilibrium states $(0,0,M)$, of the Rikitake system \eqref{rik}. Next theorem gives an alternative proof of a result from \cite{tudoranSIAM}.

\begin{theorem}\label{mainthmRiki1}
Let $(0,0,M)$, $M\in\mathbb{R}$, be an equilibrium point of the vector field $X^{\beta}$ which generates the Rikitake system \eqref{rik}. 
Then the following implications hold true:
\begin{enumerate}
\item if $|\beta|<|M|$, then the equilibrium point $(0,0,M)$ is unstable,
\item if $|\beta|>|M|$, then the equilibrium point $(0,0,M)$ is Lyapunov stable.
\end{enumerate}
\end{theorem}
\begin{proof}
The proof follows from Theorem \eqref{mainthm} applied to the Hamiltonian realization $(\mathbb{R}^3,\{\cdot,\cdot\}_{-1;C_{\beta}},H)$, and the corresponding non-degenerate regular equilibrium point $(0,0,M)$, $M\in\mathbb{R}$. More precisely, we have that
\begin{enumerate}
\item if $\mathcal{I}_{X^{\beta}}(0,0,M)=\beta^2 - M^2 <0$, or equivalently, if $|\beta|<|M|$, then the equilibrium point $(0,0,M)$ is unstable,
\item if $\mathcal{I}_{X^{\beta}}(0,0,M)=\beta^2 - M^2 >0$, or equivalently, if $|\beta|>|M|$, then the equilibrium point $(0,0,M)$ is Lyapunov stable.
\end{enumerate}

Note that the same conclusion follows if one applies the Theorem \eqref{mainthm} to the Hamiltonian realization $(\mathbb{R}^3,\{\cdot,\cdot\}_{1;C},H_{\beta})$, and the corresponding non-degenerate regular equilibrium point $(0,0,M)$, $M\neq 0$. Nevertheless, in this case we loose the information about the equilibrium $(0,0,0)$, since the origin is not a regular point of the Poisson manifold $(\mathbb{R}^3,\{\cdot,\cdot\}_{1;C})$ which generates the Hamiltonian realization $(\mathbb{R}^3,\{\cdot,\cdot\}_{1;C},H_{\beta})$.
\end{proof}

\bigskip
Next two theorems present some stability results concerning the equilibrium states $(0,0,M)$, $M\in\mathbb{R}$, of the Rikitake system \eqref{rik} restricted to the corresponding level sets of the associated first integrals. Before stating the theorems, let us recall that the system \eqref{rik} admits the Hamiltonian realizations given by $(\mathbb{R}^3,\{\cdot,\cdot\}_{1;C},H_{\beta})$, and respectively $(\mathbb{R}^3,\{\cdot,\cdot\}_{-1;C_{\beta}},H)$, where  $C(x,y,z)=H(x,y,z)=\dfrac{1}{2}\left(x^2 +y^2\right)+z^2$, and $C_{\beta} (x,y,z)=H_{\beta} (x,y,z)=\dfrac{1}{4}\left(-x^2 +y^2\right)-\beta z$, for every $(x,y,z)\in\mathbb{R}^3$.

\begin{theorem}\label{specleafrik}
Let $(0,0,M)$, $M\neq 0$, be an equilibrium point of the vector field $X^{\beta}$ which generates the Rikitake system \eqref{rik}, and let $\Sigma_{M^2}=\{(x,y,z)\in\mathbb{R}^3 : C(x,y,z) = M^2\}$ be the (regular) syplectic leaf of the Poisson manifold $(\mathbb{R}^3,\{\cdot,\cdot\}_{1;C})$, which contains $(0,0,M)$. Then the following assertions hold true.
\begin{itemize}
\item [(a)] The characteristic polynomial of $\mathfrak{L}^{X^{\beta}|_{\Sigma_{M^2}}}(0,0,M)$ is given by 
$$
p_{\mathfrak{L}^{X^{\beta}|_{\Sigma_{M^2}}}(0,0,M)}(\mu)= \mu^2 + (\beta^2 - M^2 ).
$$
\item [(b)] If $|\beta|<|M|$, then the equilibrium state $(0,0,M)$ is an unstable equilibrium point of the restricted vector field $X^{\beta}|_{\Sigma_{M^2}}$.

\item [(c)] If $|\beta|>|M|$, then the equilibrium state $(0,0,M)$ is a Lyapunov stable equilibrium point of the restricted vector field $X^{\beta}|_{\Sigma_{M^2}}$.

\item [(d)] If $|\beta|>|M|$, then there exists $\varepsilon_{0}>0$ and a one-parameter family of periodic orbits of $X^{\beta}|_{\Sigma_{M^2}}$ (and hence of $X^{\beta}$ too), $\left\{\gamma_{\varepsilon}\right\}_{0<\varepsilon\leq\varepsilon_0}\subset \Sigma_{M^2}$, that shrink to $(0,0,M)$ as $\varepsilon\rightarrow 0$, with periods $T_{\varepsilon}\rightarrow{\frac{2\pi}{\sqrt{\beta^2 - M^2}}}$ as $\varepsilon\rightarrow 0$. Moreover, the set $\{(0,0,M)\}\cup\bigcup_{0<\varepsilon < \varepsilon_0} \gamma_{\varepsilon}$ represents the connected component of $\Sigma_{M^2}\setminus \gamma_{\varepsilon_{0}}$, which contains the equilibrium point $(0,0,M)$.
\end{itemize}
\end{theorem}
\begin{proof}
The proof follows directly from Theorem \eqref{specleaf} applied to the Hamiltonian realization $(\mathbb{R}^3,\{\cdot,\cdot\}_{1;C},H_{\beta})$ of the the Rikitake system \eqref{rik}, and the associated non-degenerate regular equilibrium state $(0,0,M)$, $M\neq 0$.
\end{proof}

\begin{theorem}\label{specleafrik2}
Let $(0,0,M)$, $M\in\mathbb{R}$, be an equilibrium point of the vector field $X^{\beta}$ which generates the Rikitake system \eqref{rik}, and let $\Sigma^{\beta}_{-\beta M}=\{(x,y,z)\in\mathbb{R}^3 : C_{\beta}(x,y,z) = -\beta M\}$ be the (regular) syplectic leaf of the Poisson manifold $(\mathbb{R}^3,\{\cdot,\cdot\}_{-1;C_{\beta}})$, which contains $(0,0,M)$. Then the following assertions hold true.
\begin{itemize}
\item [(a)] The characteristic polynomial of $\mathfrak{L}^{X^{\beta}|_{\Sigma^{\beta}_{-\beta M}}}(0,0,M)$ is given by 
$$
p_{\mathfrak{L}^{X^{\beta}|_{\Sigma^{\beta}_{-\beta M}}}(0,0,M)}(\mu)= \mu^2 + (\beta^2 - M^2 ).
$$
\item [(b)] If $|\beta|<|M|$, then the equilibrium state $(0,0,M)$ is an unstable equilibrium point of the restricted vector field $X^{\beta}|_{\Sigma^{\beta}_{-\beta M}}$.

\item [(c)] If $|\beta|>|M|$, then the equilibrium state $(0,0,M)$ is a Lyapunov stable equilibrium point of the restricted vector field $X^{\beta}|_{\Sigma^{\beta}_{-\beta M}}$. 

\item [(d)] If $|\beta|>|M|$, then there exists $\widetilde{\varepsilon_{0}}>0$ and a one-parameter family of periodic orbits of $X^{\beta}|_{\Sigma^{\beta}_{-\beta M}}$ (and hence of $X^{\beta}$ too), $\left\{\widetilde{\gamma_{\varepsilon}}\right\}_{0<\varepsilon\leq\widetilde{\varepsilon_{0}}}\subset \Sigma^{\beta}_{-\beta M}$, that shrink to $(0,0,M)$ as $\varepsilon\rightarrow 0$, with periods $T_{\varepsilon}\rightarrow{\frac{2\pi}{\sqrt{\beta^2 - M^2}}}$ as $\varepsilon\rightarrow 0$. Moreover, the set $\{(0,0,M)\}\cup\bigcup_{0<\varepsilon < \widetilde{\varepsilon_{0}}} \widetilde{\gamma_{\varepsilon}}$ represents the connected component of  $\Sigma^{\beta}_{-\beta M}\setminus \widetilde{\gamma_{\widetilde{\varepsilon_{0}}}}$, which contains the equilibrium point $(0,0,M)$.
\end{itemize}
\end{theorem}
\begin{proof}
The proof follows directly from Theorem \eqref{specleaf} applied to the Hamiltonian realization $(\mathbb{R}^3,\{\cdot,\cdot\}_{-1;C_{\beta}},H)$ of the the Rikitake system \eqref{rik}, and the associated non-degenerate regular equilibrium state $(0,0,M)$, $M\in\mathbb{R}$. Note that, in contrast to Theorem \eqref{specleafrik}, in this case we obtain also the existence of periodic orbits around the origin. 
\end{proof}

%\subsection*{Acknowledgment}
%This work was supported by a grant of the Romanian National Authority for Scientific Research, CNCS-UEFISCDI, project number PN-II-RU-TE-2011-3-0103. 

\bigskip
\bigskip

\noindent {\sc R.M. Tudoran}\\
West University of Timi\c soara\\
Faculty of Mathematics and Computer Science\\
Department of Mathematics\\
Blvd. Vasile P\^arvan, No. 4\\
300223 - Timi\c soara, Rom\^ania.\\
E-mail: {\sf tudoran@math.uvt.ro}\\

\end{document}